\documentclass[a4paper,10pt, twoside]{scrartcl} 
\DeclareOldFontCommand{\rm}{\normalfont\rmfamily}{\mathrm}
\DeclareOldFontCommand{\sf}{\normalfont\sffamily}{\mathsf}
\DeclareOldFontCommand{\tt}{\normalfont\ttfamily}{\mathtt}
\DeclareOldFontCommand{\bf}{\normalfont\bfseries}{\mathbf}
\DeclareOldFontCommand{\it}{\normalfont\itshape}{\mathit}
\DeclareOldFontCommand{\sl}{\normalfont\slshape}{\@nomath\sl}
\DeclareOldFontCommand{\sc}{\normalfont\scshape}{\@nomath\sc}
\usepackage[english]{babel}
\usepackage[utf8]{inputenc}
\usepackage{cite}
\usepackage{leftidx,tensor}
\usepackage{amsmath}
\usepackage{amsthm}
\usepackage{mathtools}\usepackage{amsfonts}
\usepackage{amssymb}
\usepackage{fancyhdr}
\usepackage{tabularx}
\usepackage{geometry}
\usepackage{enumitem}
\usepackage{layout}
\usepackage{subfig}
\usepackage{floatflt}
\usepackage[T1]{fontenc}
\usepackage{bbm}
\usepackage[usenames,dvipsnames]{color}
\usepackage{titlesec}
\usepackage[subfigure,titles]{tocloft}
\usepackage[pdfpagelabels=true, colorlinks, linkcolor = blue , citecolor = blue, filecolor = blue, urlcolor = blue]{hyperref}

\makeatletter
\def\l@lstlisting#1#2{\@dottedtocline{1}{0em}{1em}{\hspace{1,5em} Lst. #1}{#2}}
\makeatother
\DeclareMathAlphabet\mathcalbm{OMS}{cmsy}{b}{n}

\newtheoremstyle{normal}
{10pt}
{10pt}
{}
{}
{\bfseries}
{}
{0em}
{\bfseries{\thmname{#1}\thmnumber{ #2}\thmnote{\hspace{0em}(#3)\newline}}}

\newtheoremstyle{standard}  
  {10pt}   
  {}   
  {\itshape}  
  {}       
  {\bfseries} 
  {:}         
  {0.2cm}  
  {\bfseries{\thmname{#1}\thmnumber{ #2}\thmnote{ \hspace{0em}(#3)}}}          

\newtheoremstyle{mittitel}  
  {10pt}   
  {}   
  {\itshape}  
  {}       
  {\bfseries} 
  {:}         
  {0.2cm}  
  {\bfseries{\thmname{#1}\thmnumber{ #2}\thmnote{ \hspace{0em}(#3)\newline}}}          

\DeclareMathOperator{\s}{\mathcal{S}}
\DeclareMathOperator{\D}{\mathcal{D}}

\DeclareMathOperator{\ii}{\textit{\textbf{i}}}
\DeclareMathOperator{\jj}{\textit{\textbf{j}}}

\newcommand{\supp}{\mathop{\mathrm{supp}}}

\begin{document}
\frenchspacing

\thispagestyle{empty}

\begin{center}
\Large{\textbf{SPECTRAL ASYMPTOTICS FOR KREIN-FELLER-OPERATORS WITH RESPECT TO $\boldsymbol V$-VARIABLE CANTOR MEASURES}}
\\[15pt]
\normalsize{LENON A. MINORICS\footnote{ Institute of Stochastics and Applications, University of Stuttgart, Pfaffenwaldring 57, 70569 Stuttgart, Germany, Email: Lenon.Minorics@mathematik.uni-stuttgart.de}}
\end{center}

\titleformat{\section}{\large\filcenter\scshape}{}{1em}{\thesection. \hspace{0em}}
\titleformat{\subsection}[runin]{\bfseries}{}{0pt}{\thesubsection. \hspace{0em}}
\titleformat{name=\section,numberless}[block]{\large\scshape\centering}{}{0pt}{}

\theoremstyle{standard}
 
\newtheorem{thm}{Theorem}[section]
\newtheorem{prop}[thm]{Proposition} 
\newtheorem{lem}[thm]{Lemma}
\newtheorem{kor}[thm]{Corollary} 
\newtheorem{defi}[thm]{Definition} 
\newtheorem{const}[thm]{Construction}
\newtheorem{remark}[thm]{Remark}
\newtheorem{example}[thm]{Example}
\newtheorem{condition}[thm]{Condition}
\newtheorem{assumption}[thm]{Assumption}
\newtheorem{acknowledgement}[thm]{Acknowledgement}

\vspace*{8pt}
\textbf{Abstract.} We study the limiting behavior of the Dirichlet and Neumann eigenvalue counting function of generalized second order differential operators $\frac{d}{d \mu} \frac{d}{d x}$, where $\mu$ is a finite atomless Borel measure on some compact interval $[a,b]$. Therefore, we firstly recall the results of the spectral asymptotics for these operators received so far. Afterwards, we make a proposition about the convergence behavior for so called random $V$-variable Cantor measures.

\section{Introduction}
It is well known that $f \in C^0([a,b], \mathbb{R})$ possesses a $L_2$ weak derivative $g \in \mathcal{L}_2(\lambda^1,[a,b])$, where $\lambda^1$ denotes the one dimensional Lebesgue measure, if and only if 
\begin{align*}
f(x) = f(a) + \int_a^x g(y) \, d y.
\end{align*}
Replacing the one dimensional Lebesgue measure by some measure $\mu$ leads to a generalized $L_2$ weak derivative depending on the measure $\mu$. Therefore, we let $\mu$ be a finite non-atomic Borel measure on some interval $[a,b]$, $-\infty < a < b < \infty$. 
The \textit{$\mu$-derivative} of $f: [a,b] \longrightarrow \mathbb{R}$ for which $f^\mu \in \mathcal{L}_2(\mu)$ exists such that
\begin{align*}
f(x) = f(a) + \int_a^x f^\mu(y) \, d\mu(y) ~~~ \text{ for all } x \in [a,b]
\end{align*}
is defined as the unique equivalence class of $f^\mu$ in $L_2(\mu)$. We denote this equivalence class by $\frac{d f}{ d \mu}$. The \textit{Krein-Feller-operator} $\frac{d}{d \mu} \frac{d}{d x} f$ is than given as the $\mu$-derivative of the $\lambda^1 _{|_{[a,b]}}$-\\derivative of $f$. 

This operator were introduced for example in \cite{Ito65}. \cite{Kue80}, \cite{Kue86}, \cite{Loe91}, \cite{Loe93} investigate on properties of the generated stochastic process, called quasi or gap diffusion, and related objects.
\\

As in e.g. \cite{Arz14}, \cite{Fuj87}, we are interested in the spectral asymptotics for generalized second order differential operators $\frac{d}{d \mu} \frac{d}{d x}$ with Dirichlet or Neumann boundary conditions, i.e. we study the equation
\begin{align}
\frac{d}{d \mu} \frac{d}{d x} f = - \lambda f \label{eigenequ}
\end{align}
with
\begin{align*}
f(a) = f(b) = 0 ~~~~~ \text{or} ~~~~~ f'(a) = f'(b) = 0. 
\end{align*}

For a physical motivation, we consider a flexible string which is clamped between two points $a$ and $b$. If we deflect the string, a tension force drives the string back towards its state of equilibrium. Mathematically, the deviation of the string is described by some solution $u$ of the one dimensional wave equation
\begin{align*}
\frac{\rho(x)}{F} \frac{\partial^2 u(t,x)}{\partial t^2} = \frac{\partial^2 u(t,x)}{\partial x^2}, ~~~ x \in [a,b], ~ t \in [0,\infty)
\end{align*}
with Dirichlet boundary condition $u(t,a) = u(t,b) = 0$ for all $t$. Hereby, $\rho$ is given as the density of the mass distribution of the string and $F$ as the tangential acting tension force. To solve this equation, we make the ansatz $u(t,x) = \psi(t) \, \phi(x)$ and receive 
\begin{align*}
\frac{\psi '' (t)}{F \, \psi(t)} = \frac{\phi''(x)}{\phi (x) \, \rho(x)} = - \lambda,
\end{align*}
for some constant $\lambda \in \mathbb{R}$. In the following, we only consider the equation

\begin{align*}
\frac{\phi''(x)}{\phi  (x) \rho(x)} = - \lambda.
\end{align*}
Thus, we have
\begin{align*}
\phi'(t) - \phi '(a) = - \lambda \int_a^t \phi(y) \, d\mu (y),
\end{align*}
where $\mu$ is the mass distribution of the string. In other words,
\begin{align}
\frac{d}{d \mu} \frac{d}{d x} \phi = - \lambda \, \phi. \label{eigequ}
\end{align}
This equation no longer involves the density $\rho$, meaning that we can reformulate the problem for singular measures $\mu$.  Such a solution $\phi$ can be regarded as the shape of the string at some fixed time $t$. Up to a multiplicative constant, the natural frequencies of the string are given as the square root of the eigenvalues of \eqref{eigequ}.

In Freiberg \cite{Fre03} analytic properties of this operator are developed. There, it is shown that $-\frac{d}{d \mu} \frac{d}{d x}$ with Dirichlet or Neumann boundary conditions has a pure point spectrum and no finite accumulation points. Moreover, the eigenvalues are non-negative and have finite multiplicity. \\ We denote the sequence of Dirichlet eigenvalues of $ -\frac{d}{d \mu} \frac{d}{d x}$ by $\left(\lambda^\mu_{D,n}\right)_{n \in \mathbb{N}}$ and the sequence of Neumann eigenvalues by $\left(\lambda^\mu_{N,n}\right)_{n \in \mathbb{N}_0}$, where we assort the eigenvalues ascending and count them according to multiplicities. Let 
\begin{align*}
N_D^\mu (x) \coloneqq \# \left \{i \in \mathbb{N}: ~ \lambda^\mu_{D,i} \le x \right\} ~~~ \text{ and } ~~~ N_N^\mu(x) \coloneqq \# \left \{i \in \mathbb{N}_0: ~ \lambda^\mu_{N,i} \le x \right\}.
\end{align*}
$N_D^\mu$ and $N_N^\mu$ are called the Dirichlet and Neumann eigenvalue counting function of $-\frac{d}{d \mu} \frac{d}{d x}$, respectively. 
The problem of determining $\gamma > 0$ such that
\begin{align}
N_{D/N}^\mu(x) \asymp x^\gamma, ~~~ x \rightarrow \infty,
\end{align}
is an extension of the analogous problem for the one dimensional Laplacian. The following theorem is a well-known result of Weyl \cite{Wey15}.
\begin{thm} \label{weyls law}
Let $\Omega \subseteq \mathbb{R}^n$ be a domain with smooth boundary $\partial \Omega$. Consider the eigenvalue problem
\begin{align*}
\begin{cases}
- \Delta_{n,\Omega} u &= \lambda u ~ \text{ on } ~\Omega, \\
u_{|\partial\Omega} &= 0,
\end{cases}
\end{align*}
where $\Delta_{n,\Omega}$ denotes the Laplace operator on $\Omega$. Then, for the Dirichlet eigenvalue counting function $N_D^{({n,\Omega})}$ of $\Delta_{n,\Omega}$ it holds that
\begin{align}
N_D^{({n,\Omega})}(x) = (2\pi)^{-n} \, c_n \, \textit{vol}_n(\Omega) \, x^{n/2} + o\left(x^{n/2}\right), ~~~ x\rightarrow \infty, \label{weyls formula}
\end{align}
hereby $c_n$ denotes the volume of the $n$-dimensional unit ball.
\end{thm}
Choosing $\mu = \lambda_{|_{[a,b]}}^1$ leads to
\begin{align*}
N^{\mu}_{D}(x) = N_D^{(1,(a,b))}(x) \asymp x^{1/2}, ~~~ x \rightarrow \infty,
\end{align*}
which gives the leading order term in the Weyl asymptotics as in Theorem \ref{weyls law}.

 \eqref{weyls formula} motivates the definition of the spectral dimension 
\begin{align}
\frac{d_s(\Omega)}{2} \coloneqq \lim_{\lambda \rightarrow \infty} \frac{\log N_D^{(n,\Omega)}(\lambda)}{\log \lambda}. \label{spec dim}
\end{align}
Which leads to 
\begin{align*}
d_s(\Omega) = n
\end{align*}
in Theorem \ref{weyls law}.
Many authors before studied the expression \eqref{spec dim} for generalized Laplacians on p.c.f. fractals, e.g. \cite{Fre15}, \cite{Ham00}.
In this paper, we investigate on this expression for the Krein-Feller-operator on so called $V$-variable Cantor sets. Therefore, we call the limit
\begin{align*}
\gamma \coloneqq \gamma(\mu) \coloneqq \lim_{\lambda \rightarrow \infty} \frac{\log N_D^\mu(\lambda)}{\log \lambda}
\end{align*}
the \textit{spectral exponent} of the corresponding Krein-Feller-operator.
\\\\
$V$-variable Cantor measures interpolate between homogeneous and recursive Cantor measures. In the homogeneous case, we take in every approximation step one iterated function system and split each interval of the  previous approximation step according to this IFS. In the recursive case we do allow to take arbitrary IFSs of the given setting for an interval, independent of the IFSs used for intervals of the same construction level. Now, in the $V$-variable case, we allow in every approximation step to take $V \in \mathbb{N}$ IFSs. For $V=1$ this reduces to the homogeneous case and as $V$ tends to infinity we receive in the limit the recursive case.    
\\\\
As an example of the different types of fractals, we take four different iterated function systems $S^{(1)}$, $S^{(2)}$, $S^{(3)}$ and $S^{(4)}$ on the unit interval $[0,1]$ under consideration. We let $S^{(1)}$ be the generator of the Cantor set, $S^{(2)}$ be the IFS consisting of three linear functions which split the unit interval into five parts such that the second and fourth open fifth intervals are removed, $S^{(3)}$ be the IFS consisting of two linear functions such that the unit interval is split into three parts, where the second open fourth interval is removed and $S^{(4)}$ be the IFS consisting of two linear functions such that the unit interval is split into three parts, where the third open forth interval is removed.  The first approximation steps of one possible homogeneous Cantor set corresponding to this setting are shown in figure \ref{homCantor}.

\begin{figure}[ht]
\includegraphics[scale=0.38]{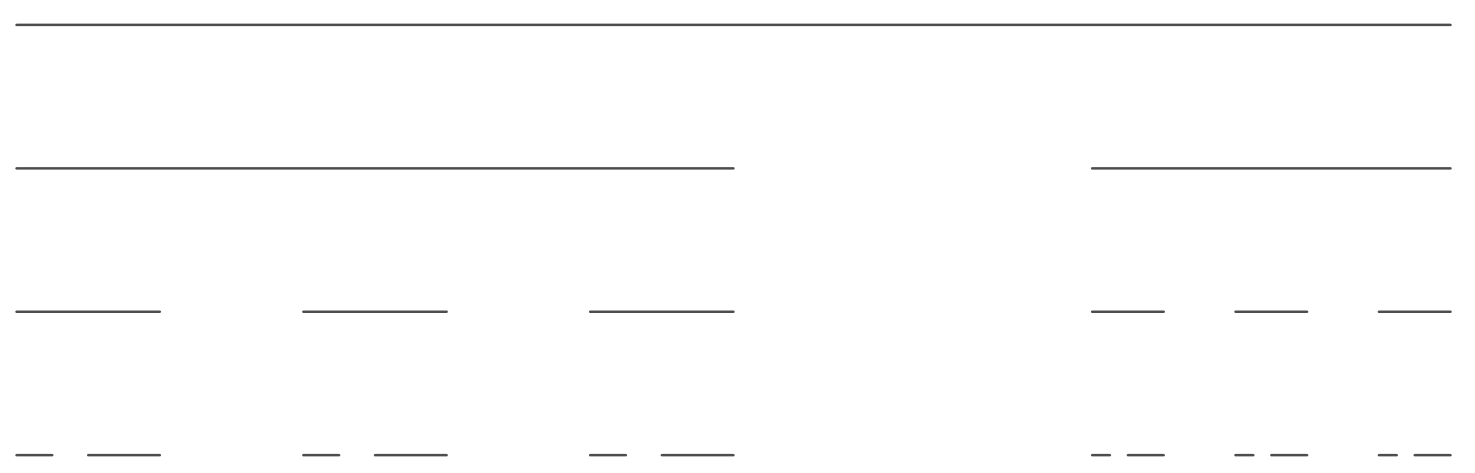}
\caption{First three approximation steps of one possible homogeneous Cantor set constructed by the sequence of indices $4,2,1$}
\label{homCantor}
\end{figure}

As shown in the figure, in the homogeneous case we split the remaining intervals in an approximation step according to one IFS indicated by one of our indices, where our index set in this example is $\{1,2,3,4\}$. For the recursive case, we allow to split every interval according to different iterated function systems, even in the same approximation step. Therefore, we totally destroy every symmetry in the fractal.

\begin{figure}[ht]
\includegraphics[scale=0.4]{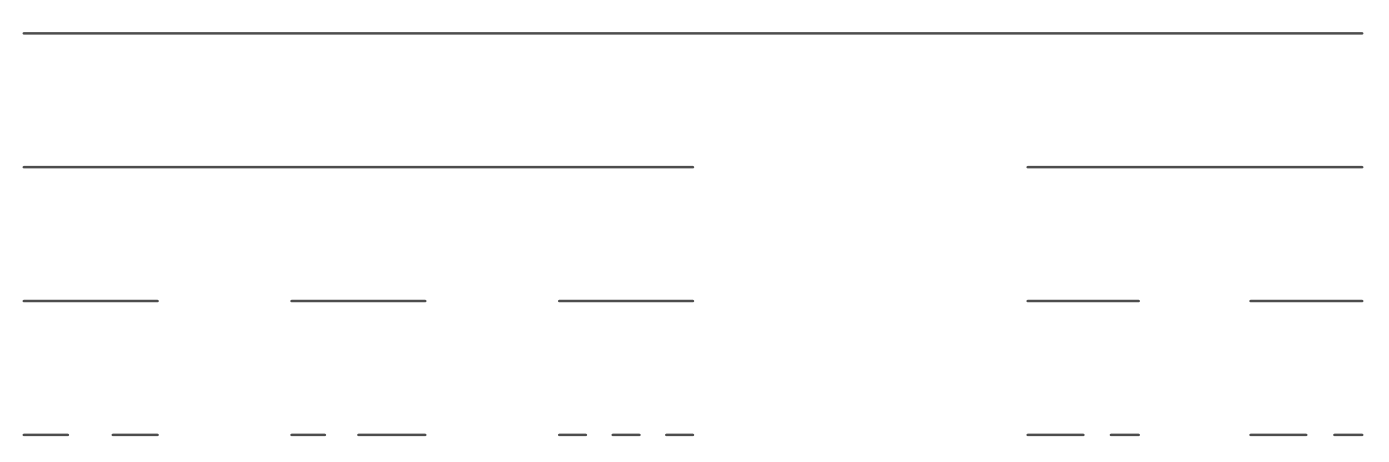} \\[5pt]

\hspace{58pt}\includegraphics[width=0.7\textwidth, height=235px]{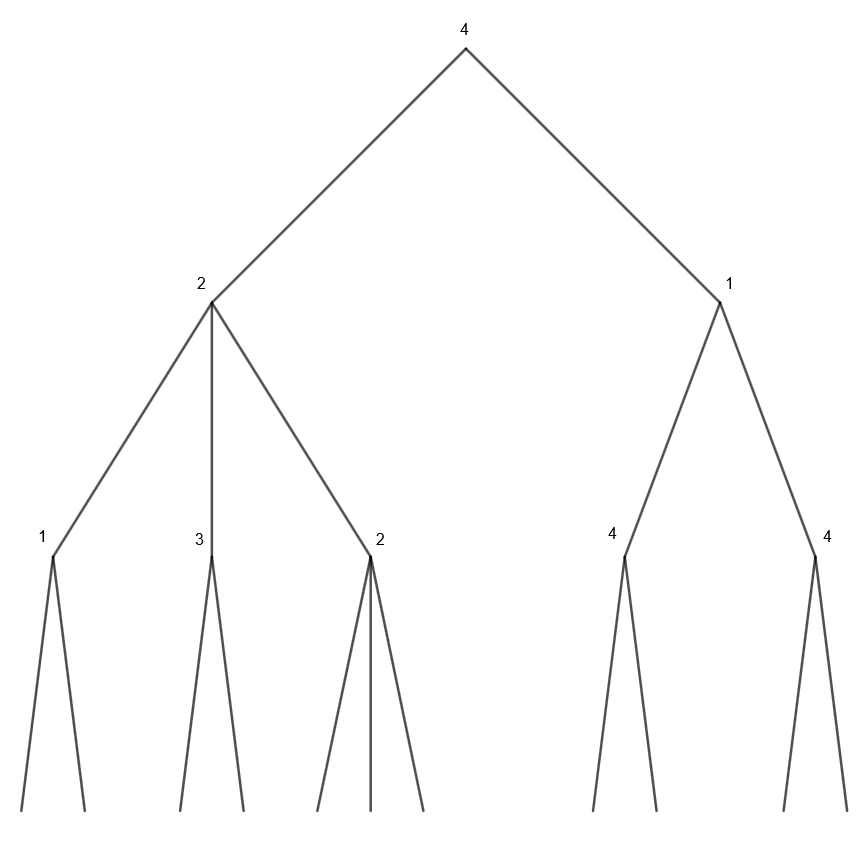}
\caption{First three approximation steps of one possible recursive Cantor set and corresponding construction tree}
\label{recCantor}
\end{figure}

As shown in figure \ref{recCantor}, we code the construction in a labelled tree, as will be explained in Chapter \ref{self sim and hom case}. These trees are also used to code the construction of $V$-variable Cantor sets.

As an example of a $V$-variable Cantor set, let be $V=3$. This indicates the number of types, where we denote the three different types by $\triangledown$, $\square$, $\lozenge$. In every approximation step, every type indicates an index of our index set $\{1,2,3,4\}$. The indicated index of a particular type can differ in different approximation steps. The following figure shows how we construct a $3$-Variable Cantor set in this setting. The fractal depends on a sequence of so called environments which determine in every step the indicated indexes of each type and also the types of the intervals in the next step.

\begin{figure}[ht]

\includegraphics[scale=0.36]{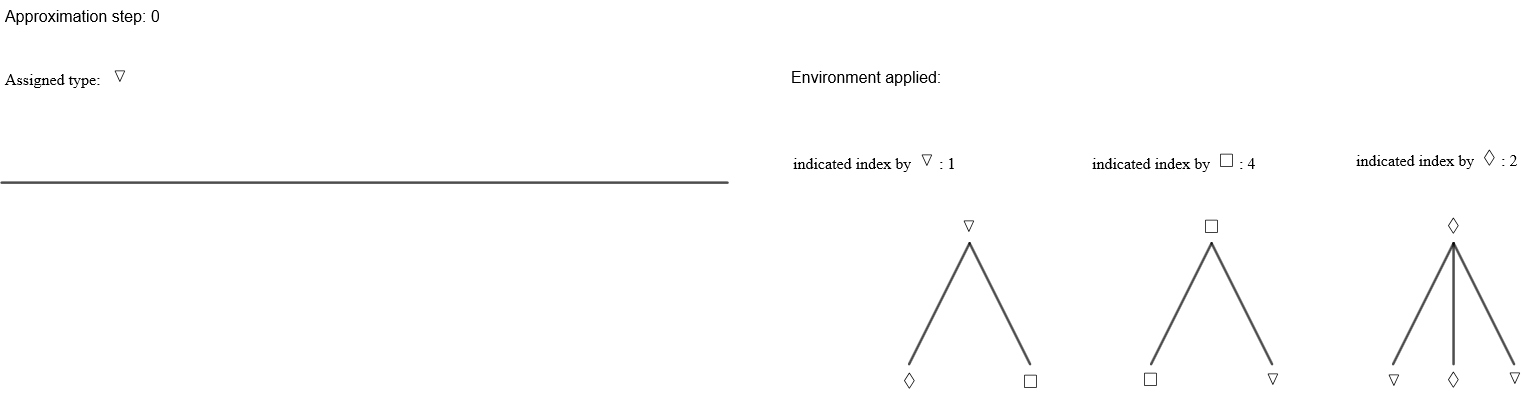}
\includegraphics[scale=0.36]{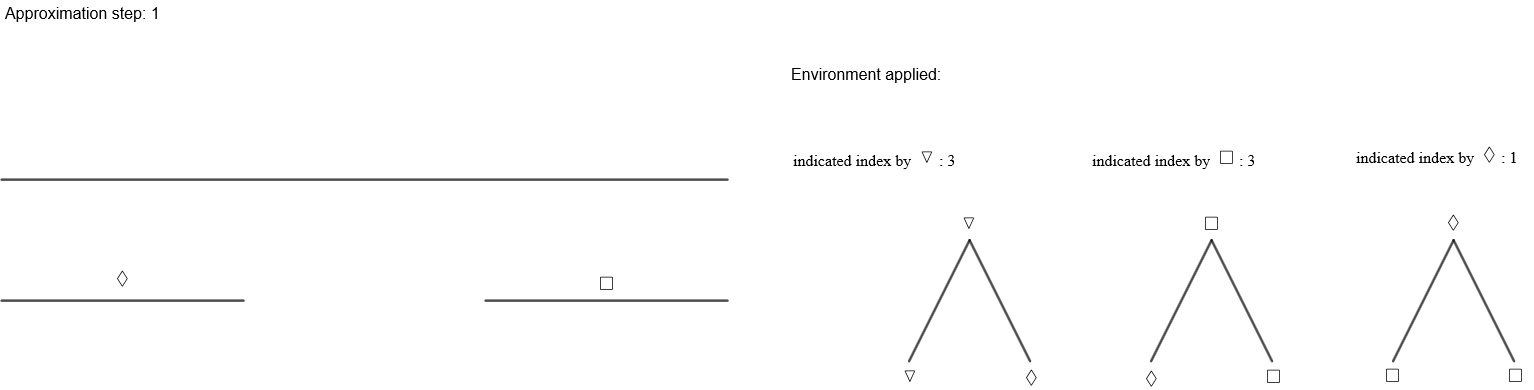}
\includegraphics[scale=0.36]{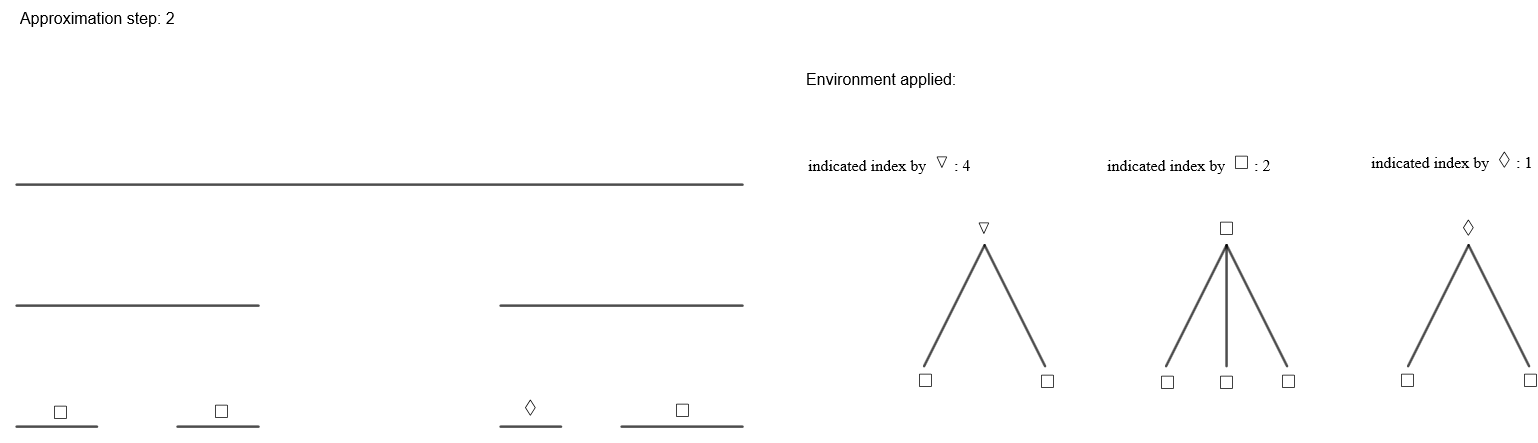}
\caption{First approximation steps of one possible $3$-variable Cantor set}
\label{VCantor}
\end{figure}

Remark that the number of usable iterated function systems in the $V$-variable case in a particular approximation step is not only bounded by the number of indices (as in the recursive case), but also by $V$. After applying the environment, in approximation step 2 of figure \ref{VCantor}, all assigned types are equal. In the random case, such levels will occur infinitely often almost surely and will be crucial for our consideration. We call such levels \textit{necks} and discuss some properties in Chapter \ref{necks}.

We are interested in the spectral asymptotics of $V$-variable Cantor measures, which are natural extensions of self similar Cantor measures on $V$-variable Cantor sets. More precisely,
we consider the asymptotic behavior of $N_D^\mu(x)$ and $N_N^\mu(x)$ as $x$ tends to infinity for so called random $V$-variable Cantor measures $\mu$.
The spectral asymptotics for Krein-Feller-operators with respect to self similar measures was developed by Freiberg \cite{Fre05}, with respect to random (and deterministic) homogeneous Cantor measures by Arzt \cite{Arz14} and w.r.t. random recursive Cantor measures in \cite{Min17}.

The paper is organized as follows. In Chapter \ref{prelim} we give the definition of the operator which is under consideration and recap the important results received so far. Then, we give in Chapter \ref{Const V} firstly the definition of the $V$-variable Cantor sets and measures and discuss afterwards the important neck levels. Also in this Chapter, we give the definition of so called cut sets. A sequence of Cut sets, related to the neck levels, will then be used in Chapter \ref{section asympotic} to give the spectral asymptotics. To this end, we start in Chapter \ref{section asympotic} by giving the Dirichlet-Neumann-Bracketing with which we receive upper and lower bounds for the eigenvalue counting functions. These bounds will finally help us to determine the spectral exponent.

\section{Preliminaries} \label{prelim}
\subsection{Definition of the Krein-Feller-Operator.} \label{ablmu}
Let $\mu$ be a finite non-atomic Borel measure on $[a,b]$, $-\infty < a < b < \infty$ and 
\begin{align*}
\D_1^\mu \coloneqq \bigg\{f: [a,b] \longrightarrow \mathbb{R}: ~ \exists ~ f^\mu \in ~&\mathcal{L}_2(\mu): \\& f(x) = f(a) + \int_a^x f^{\mu}(y) ~ d \mu (y), ~~~ x \in [a,b] \bigg\}.
\end{align*}
The $\mu$\textit{-derivative} of $f$ is defined as the equivalence class of $f^\mu$ in $L_2(\mu)$. It is known (see \cite[Corollary 6.4]{Fre03}) that this equivalence class is unique. Thus, the operator 
\begin{align*}
\frac{d}{d \mu} : D_1^\mu &\longrightarrow L_2(\mu), \\
f &\mapsto ~~ [f^\mu]_{\sim_\mu} 
\end{align*}
is well-defined.
Let
\begin{align*}
\D \coloneqq \D_2^{\mu , \lambda^1}\coloneqq \bigg\{ f \in C^1((a,b)) \cap C^0([a,b]) :& ~ \exists ~ (f')^\mu \in \mathcal{L}_2\left( \lambda^1_{|_{[a,b]}} \right): \\& f'(x) = f'(a) + \int_a^x (f')^\mu(y) ~  d \mu (y), ~~~ x \in [a,b]  \bigg\}.
\end{align*}
The \textit{Krein-Feller-operator w.r.t.} $\mu$ is given as 
\begin{align*}
\frac{d}{d \mu} \frac{d}{d x}: \D &\longrightarrow  L_2(\mu) \\
f &\mapsto ~~ [(f')^\mu]_{\sim_\mu}.
\end{align*}

\subsection{Spectral Asymptotics for Self-Similar, Random Homogeneous and Random Recursive Cantor Measures.} \label{self sim and hom case}
As mentioned in the introduction, the spectral asymptotics for Krein-Feller-operators were discovered by \cite{Fre05} and \cite{Arz14} for special types of measures. In this section we summarize their main results.  Firstly, we consider self-similar measures, treated in \cite{Fre05}. Therefore, let $\s = \{S_1,...,S_N\}$ be an iterated function system given by
\begin{align*}
S_i(x) = r_i \, x + c_i, ~~~ x \in [a,b],
\end{align*}
whereby $r_i \in (0,1)$, $c_i \in \mathbb{R}$ are constants such that the open set condition is fulfilled, $S_i[a,b] \subseteq [a,b]$ for all $i$ and let $m = (m_1,...,m_N)$ be a vector of weights. As shown in \cite{Hut81}, there exists a unique non-empty compact set $C=C(\s) \subseteq [a,b]$ such that $\bigcup_{i=1}^N S_i(C) = C$ and a unique Borel probability measure $\mu = \mu(\s,m)$ such that $\mu = \sum_{i=1}^N m_i \, \mu \circ S_i^{-1}$. Moreover it holds $\supp \mu = C$. We call $C$ self-similar w.r.t. to $\s$ and $\mu$ self-similar w.r.t. $\s$ and $m$. The Hausdorff dimension of $C$ is given by the unique solution $d \in [0,1]$ of $\sum_{i=1}^N r_i^{d} = 1$ and it holds $\mathcal{H}^{d}(C) \in (0,\infty)$. Moreover, if $m_i = r_i^{d}$ for all $i$, we have $\mu = \mathcal{H}^d(C)^{-1} \, \mathcal{H}^d_{|_{C}}$. In this setting, the spectral exponent of the corresponding Krein-Feller-operator is the unique solution $\gamma > 0$ of $\sum_{i=1}^N \left( m_i \, r_i \right)^\gamma = 1$. The spectral exponent were discovered by \cite{Fuj87} and more general by \cite[Theorem 4.1]{Fre05}. 
\\\\
To recap the results of \cite[Section 3]{Arz14}, let $J$ be a non-empty countable set. To each $j \in J$ we define an iterated function system $\s^{(j)}=\left\{S_1^{(j)},...,S_{N_j}^{(j)} \right\}$, $N_j \in \mathbb{N}$ such that 
\begin{align*}
S_i^{(j)}(x) = r_i^{(j)} \, x + c_i^{(j)}, ~~~ x \in [a,b], ~ i=1,...,N_j,
\end{align*}
where the constants $r_i^{(j)} \in (0,1)$, $c_i^{(j)} \in \mathbb{R}$ are chosen such that
\begin{align} \label{assumption IFS}
a = S_1^{(j)}(a) < S_1^{(j)} (b) \le S_2^{(j)}(a) < \cdots < S_{N_j}^{(j)}(b) = b. 
\end{align}
Further, we call $\xi = \left(\xi_1, \xi_2 ,... \right)$, $\xi_i \in J$ an environment sequence and define 
\begin{align*}
W_n \coloneqq \left\{1 ,..., N_{\xi_1} \right\} \times  \left\{1 ,..., N_{\xi_2} \right\} \times \cdots \times \left\{1 ,..., N_{\xi_n} \right\}, ~~~ n \in \mathbb{N}. 
\end{align*} 
The homogeneous Cantor set to a given environment sequence $\xi$ is
\begin{align*}
K^{(\xi)} \coloneqq \bigcap_{n=1}^\infty \bigcup_{w \in W_n} \left(S_{w_1}^{(\xi_1)} \circ S_{w_2}^{(\xi_2)} \circ \cdots \circ S_{w_n}^{(\xi_n)} \right)([a,b]). 
\end{align*}
Next, we define a measure $\mu^{(\xi)}$ on $[a,b]$ to a given environment sequence, which generalizes the invariant measures, presented before. To this end, let $m^{(j)}= (m_1^{(j)},...,m_{N_j}^{(j)})$, $j \in J$ be a vector of weights. $\mu^{(\xi)}$ is defined as the week limit of the sequence of Borel probability measures $\left(\mu_n^{(\xi)}\right)_{n \in \mathbb{N}}$,
\begin{align*}
\mu_n^{(\xi)} \coloneqq \sum_{w \in W_n} m_{w_1}^{(\xi_1)} \cdots m_{w_n}^{(\xi_n)} \, \mu_0 \circ \left(S_{w_1}^{(\xi_1)} \circ \cdots \circ S_{w_n}^{(\xi_n)} \right)^{-1}, ~~~~~ \mu_0 \coloneqq \frac{1}{b-a} \, \lambda^1_{|_{[a,b]}}.
\end{align*}
$\mu^{(\xi)}$ is called homogeneous Cantor measure, corresponding to $K^{(\xi)}$. If $|J| = 1$, then the definition of invariant sets and measures coincide with $K^{(\xi)}$ and $\mu^{(\xi)}$.\\
\cite[Theorem 3.3.10]{Arz14} makes a statement about the spectral exponent of the Krein-Feller-operator with respect to $\mu^{(\xi)}$, where $\xi$ is a deterministic environment sequence. Here, we only consider the random case. This means, the sequences $\xi$ are i.i.d. random variables. Therefore, let $(\Omega,\mathcal{F},\mathbb{P})$ be a probability space and $\xi = (\xi_1,\xi_2,..)$ a sequence of i.i.d. $J$-valued random variables with $p_j \coloneqq \mathbb{P}(\xi_i = j)$. We denote the Dirichlet and Neumann eigenvalue counting function of the Krein-Feller-operator w.r.t. $\mu^{(\xi(\omega))}$ by $N_D^{(\xi(\omega))}$ and $N_N^{(\xi(\omega))}$, respectively. Further, let the following five conditions be satisfied:
\begin{enumerate}[label=\textbf{A.\arabic*}]
\item[] 
\begin{align}
\max_{j \in J} N_j < \infty, \label{A1} \tag{A1}
\end{align} 
\item[] 
\begin{align}
\inf_{j \in J} \min_{i=1,...,N_j} r_i^{(j)} m_i^{(j)} > 0, \label{A2} \tag{A2}
\end{align}
\item[]
\begin{align}
\sup_{j \in J} \max_{i=1,...,N_j} r_i^{(j)} m_i^{(j)} < \infty, \label{A3} \tag{A3}
\end{align}
\item[] 
\begin{align} 
\prod_{j \in J, \atop \sum_{i=1}^{N_j} \left(r_i^{(j)} m_i^{(j)}\right)^\gamma < 1} \sum_{i=1}^{N_j} \left(r_i^{(j)} m_i^{(j)}\right)^\gamma > 0, \label{A4} \tag{A4}
\end{align}
\item[]
\begin{align}
\prod_{j \in J, \atop \sum_{i=1}^{N_j} \left(r_i^{(j)} m_i^{(j)}\right)^\gamma > 1} \sum_{i=1}^{N_j} \left(r_i^{(j)} m_i^{(j)}\right)^\gamma < \infty, \label{A5} \tag{A5}
\end{align}
\end{enumerate}
whereby $\gamma > 0$ is the unique solution of $\prod_{j \in J} \left( \sum_{i=1}^{N_j} \left(r_i^{(j)} m_i^{(j)}\right)^{\gamma} \right)^{p_j} = 1$. \\[5pt]
Under these assumptions, we obtain:
\begin{thm}  \label{hom asym}
Let $\gamma_h > 0$ be the unique solution of
\begin{align*}
\prod_{j \in J} \left( \sum_{i=1}^{N_j} \left(r_i^{(j)} m_i^{(j)}\right)^{\gamma_h} \right)^{p_j} = 1.
\end{align*}
Then, there exist $C_1, C_2 > 0$, $x_0 > 0$ and $c_1(\omega), c_2(\omega) > 0$ such that
\begin{align*}
C_1 \, x^{\gamma_h} \, e^{-c_1(\omega)\sqrt{\log x \log \log x}} \le N_D^{\left(\xi(\omega) \right)}(x) \le  N^{\left(\xi(\omega)\right)}_N(x)  \le C_2 \, x^{\gamma_h} \, e^{ -c_2(\omega)\sqrt{\log x \log \log x}}
\end{align*}
for all $x > x_0$ almost surely.
\end{thm}
For references see \cite[Corollary 3.5.1]{Arz14}.
\\\\
For the recursive case, we take almost the same setting with the only difference that the index set $J$ has not to be countable. As in \cite{Min17}, we let $J$ be an index set and as before we define to each $j \in J$ an iterated function system $\s^{(j)}=\left\{S_1^{(j)},...,S_{N_j}^{(j)} \right\}$, $N_j \in \mathbb{N}$ such that 
\begin{align*}
S_i^{(j)}(x) = r_i^{(j)} \, x + c_i^{(j)}, ~~~ x \in [a,b], ~ i=1,...,N_j,
\end{align*}
where the constants $r_i^{(j)} \in (0,1)$, $c_i^{(j)} \in \mathbb{R}$ are chosen such that
\begin{align} \label{assumption IFS}
a = S_1^{(j)}(a) < S_1^{(j)} (b) \le S_2^{(j)}(a) < \cdots < S_{N_j}^{(j)}(b) = b. 
\end{align}
In the homogeneous case, we took in each approximation step of the fractal one iterated function system and split every interval of the previous approximation step according to that iterated function system. The difference between the homogeneous and the recursive case is that we do not take one iterated function system in a particular approximation step and split every interval in the approximation step before according to that IFS, but we allow to take for every interval a different IFS. In the homogeneous case we saved all information we needed to construct a homogeneous Canot set in a sequence. For the recursive case this is not enough since it is possible to take more than one IFS in an approximation step. But we can save the information we need in a tree. A tree $I$ is a population with an unique ancestor which we denote by $\emptyset$. This unique ancestor induces an index of our index set $J$ which we also denote by 0 for convenience. This individual is the single individual of the first generation of our population. 
The number of children of $\emptyset$ is given by $N_\emptyset$, i.e. by the number of contractions of the iterated function system to the index which is induces by $\emptyset$. The children of $\emptyset$ are denoted by $(1),\dots,(N_\emptyset)$. Analogously we proceed. Then, an individual $\ii \in I$ is denoted by $(i_1,...,i_n)$ if it is the $i_n$-th child of the $i_{n-1}$-th child of $\dots$ of the $i_1$-th child of 0 and it is of the $n+1$-th generation of the population $I$. Further, we denote the $n$-th generation of $I$ by $I_n$ and the generation of $\ii$ by $|\ii|$, which means that the generation of an individual is given by the length of the vector which identifies this individual plus one. Such a tree $I$ then induces a recursive Cantor set given by
\begin{align*}
K^{(I)} \coloneqq \bigcap_{n=1}^\infty \bigcup_{(i_1,...,i_{n+1}) \in I_n} S_{i_1}^{(\emptyset)} \circ S_{i_2}^{((i_1))} \circ \cdots \circ S_{i_{n+1}}^{((i_1,...,i_n))} ([a,b]).
\end{align*}
Then, we want to define a measure on this fractal with properties analogously to the homogeneous case. Therefore, we again define to each index $j \in J$ a vector of weights $\left(m_1^{(1)},...,m_{N_j}^{(j)}\right)$. The measure we want to define is then given by the weak limit of the sequence of Borel probability measures $\mu_n^{(I)}$ given by
\begin{align*}
\mu_n \coloneqq \sum_{(i_1,...,i_{n+1}) \in I_n}  m_{i_1}^{(\emptyset)} \cdots m_{i_{n+1}}^{((i_1,...,i_n))} \, \mu_0 \circ \left(S_{i_1}^{(\emptyset)} \circ \cdots \circ S_{i_{n+1}}^{((i_1,...,i_n))}\right)^{-1}, ~~~ \mu_0 \coloneqq \frac{1}{b-a}  \lambda^1_{|_{[a,b]}}.
\end{align*}
We denote this limit by $\mu^{(I)}$ and call it recursive Cantor measure, corresponding to $K^{(I)}$.
\\
For the random case, let $\left(\tilde{\Omega}, \tilde{\mathcal{B}}, \tilde{\mathbb{P}}\right)$ be a probability space and $\tilde{U}_{\ii}$, $: \left(\tilde{\Omega}, \tilde{\mathcal{B}}, \tilde{\mathbb{P}}\right) \longrightarrow (J,\mathfrak{B}(J)) $, $\ii \in \mathcal{G}$, whereby
\begin{align*}
\mathcal{G}_n &\coloneqq \left\{(i_1,...,i_n): i_j \in \mathbb{N},~ j=1,...,n \right\}, \\[5pt]
\mathcal{G} &\coloneqq \{0\} \cup \left( \bigcup_{n=1}^\infty \mathcal{G}_n \right),
\end{align*}
are i.i.d. $J$-valued random variables.
The probability space we are interested in is given by

\begin{align*}
(\Omega,\mathcal{B},\mathbb{P}) = \prod_{\ii \in \mathcal{G}} (\Omega_{\ii}, \mathcal{B}_{\ii}, \mathbb{P}_{\ii}),
\end{align*}
whereby $(\Omega_{\ii}, \mathcal{B}_{\ii}, \mathbb{P}_{\ii})$ are copies of $(\tilde{\Omega}, \tilde{\mathcal{B}}, \tilde{\mathbb{P}})$. We set $U_{\ii} = \tilde{U}_{\ii} \circ P_{\ii}$, $\ii \in \mathcal{G}$, where $P_{\ii}$ is the projection map onto the $\ii$-th component.
 $\omega \in \Omega$ indicates an infinite (random) tree $I(\omega)$. If $(i_1,...,i_n) = \ii \in \mathcal{G}$ and $N_{U_{(i_1,...,i_{n-1})}(\omega)} < i_n$, then in the infinite tree $I(\omega)$, the $i_n$-th child of $(i_1,...,i_{n-1})$ is never born, i.e. $\ii \notin I(\omega)$. If we refer to the Neumann/Dirichlet eigenvalue counting function, we write $N_{N/D}^{(\omega)}$ for $N_{N/D}^{(I(\omega))}$ and $\theta_{\ii} \omega$ if we mean the sub tree $\theta_{\ii} I(\omega)$ of $I(\omega)$ which is rooted at $\ii \in I(\omega)$.

Under some regularity conditions, which are basically conditions on the underlying (C-M-J) Branching process (for reference see \cite{Min17}), we receive the following theorem.
\begin{thm} \label{spec asymp rec}
The spectral exponent of the Krein-Feller-operator with respect to $\mu^{(I)}$ is almost surely given by the unique solution $\gamma_r > 0$ of
\begin{align*}
\mathbb{E}\left( \sum_{i=1}^{N_\emptyset} \left( r_i^{(\emptyset)} m_i^{(\emptyset)} \right)^{\gamma_r}  \right) =1
\end{align*}
\end{thm}

\begin{remark} 
\begin{itemize}
\item[1.]
For the recursive case, we only have a theorem about the spectral asymptotics in the random case. 
 \item[2.]
 Although the homogeneous Cantor measures are subsets of the recursive Cantor measures, Theorem \ref{spec asymp rec} makes no statement about the spectral asymptotics for the random homogeneous case since the probability that a recursive Cantor measure is homogeneous is 0.
\end{itemize}
\end{remark}

\section{$V$-Variable Cantor Sets and Measures} \label{Const V}

\subsection{Construction of Determinisitic $V$-Variable Cantor Sets and Measures.} \label{constr frac}

Let $J \neq \emptyset$ be an index set. We define to each $j \in J$ an IFS $\s^{(j)}$. Therefore, let $N_j \in \mathbb{N}$, $ N_j \ge 2$. Then $\s^{(j)} = \left(S_1^{(j)},...,S_{N_j}^{(j)}\right)$, where we define $S_i^{(j)} : [a,b] \longrightarrow [a,b]$ by
\begin{align*}
S_i^{(j)}(x) \coloneqq r_i^{(j)} \, x + c_i^{(j)}
\end{align*}
for some $r_i^{(j)} \in (0,1)$, $c_i^{(j)} \in \mathbb{R}$, $i=1,...,N_j$ such that
\begin{align*}
a = S_1^{(j)}(a) < S_1^{(j)}(b) \le S_2^{(j)}(a) < S_2^{(j)}(b) \le \cdots \le S_{N_j}^{(j)}(a) <  S_{N_j}^{(j)}(b)=b.
\end{align*}
Furthermore, let $m^{(j)} = \left(m^{(j)}_1,...,m^{(j)}_{N_j}\right)$ be a vector of weights. Thus, as in Chapter \ref{self sim and hom case}, an element of the index set $J$ identifies a tuple $\left(S^{(j)},m^{(j)}\right)$.\\\\
We need the following technical conditions for the spectral asymptotics:
\begin{enumerate}[label=\textbf{C.\arabic*}]
\item[] 
\begin{align}
\sup_{j \in J} N_j < \infty \label{C1} \tag{C1}
\end{align} 
\vspace{-1.2cm}
\item[]
\begin{align}
0 < m_{\inf} \coloneqq \inf_{j \in J} \min_{i=1,...,N_j} m_i^{(j)} \le \sup_{j \in J} \max_{i=1,\dots,N_j} m_i^{(j)} =: m_{\sup} < 1\label{C2} \tag{C2}
\end{align}
\vspace{-1.2cm}
\item[]
\begin{align}
0 < r_{\inf} \coloneqq \inf_{j \in J} \min_{i=1,...,N_j} r_i^{(j)} \le \sup_{j \in J} \max_{i=1,\dots,N_j} r_i^{(j)} =: r_{\sup} < 1 \label{C3} \tag{C3}
\end{align}
\end{enumerate}

We define $V$-variable trees as in \cite{Fre15}.
\begin{defi}
An \textit{environment} $E$ is a matrix $E = (E(1),...,E(V))$ which assigns to each $v \in\{1,...,V\}$ both an index $j_v^E \in J$ and a sequence of types $\left(\tau_{v,i}^E\right)_{i=1}^{N_{j_v^E}}$, i.e.
\begin{align*}
E(v) = \left(j_v^E,\tau_{v,1}^E\dots,\tau_{v,N_{j_v^E}}^E\right) \in J \times \{1,\dots,V\}^{N_{j_v^E}}, ~~~ v \in \{1,\dots,V\}.
\end{align*}
\end{defi}

To construct a $V$-variable tree, we take a sequence of environments $(E^k)_{k \ge 1}$ and define the $n$-th generation of the tree for $n \in \mathbb{N}_0$ as follows: 
\begin{labeling}[]{Generation 0:}
\item[Generation 0:] Every  $V$-variable tree has a unique ancestor which we denote by $\emptyset$. To this ancestor we assign a type $\tau^\emptyset$.
\item[Generation 1:] Set $v \coloneqq \tau^\emptyset$ and $ S_v \coloneqq S^{\left(j_v^{E^1}\right)}$. This determines the first IFS to be used. The number of children of the ancestor $\emptyset$ is the number of contractions of $S_v$. Assign to the $i$-th child of $\emptyset$ the type $ \tau_{v,i}^{E^1}$.
\item[Generation 2:] Repeat the procedure for generation 1 for every individual of the first generation, whereby $E^1$ is replaced by $E^2$.
\item[] \vdots
\end{labeling}
We denote a $V$-variable tree by $I_V$. Furthermore, we denote $\ii \in I_V$ by $\ii = (i_1,...,i_n)$ if it is an individual of the $n$-th generation of $I_V$ and if it is the $i_n$-th child of the $i_{n-1}$-th child of ... of the $i_1$-th child of $\emptyset$. The $n$-th generation of $I_V$ is denoted by $I_{V,n}$ and the subtree of $I_V$ rooted at $\ii$ by $\theta_{\ii} I_V$.  By construction, we have assigned to each node $\ii \in I_{V,n}$ an index $j_{\tau^{\ii}}^{E^{n}}$ and therefore a tuple consisting of an IFS $S^{(j_{\tau^{\ii}}^{E^{n}})}$ and a vector of weights $m^{(j_{\tau^{\ii}}^{E^{n}})}$ . For convenience, we denote this index also by $\ii$.

In the following, we fix a $V$-variable tree $I_V$ and suppress $V$, i.e. $I = I_V$, $I_n = I_{V,n}$.  For $\ii \in I_{n}$, $\ii = (i_1,...,i_n)$, we define
\begin{align*}
m_{\ii} &\coloneqq m_{i_1}^{(\emptyset)} \cdots m_{i_n}^{((i_1,...,i_{n-1}))}, \\
S_{\ii}([a,b]) &\coloneqq S_{i_1}^{(\emptyset)} \circ ... \circ S_{i_{n}}^{((i_1,...,i_{n-1}))} ([a,b])
\end{align*}
and we define analogously $S_{\ii}^{-1}([a,b])$ as the composition of the preimages.  With these notations, we can easily transfer the definition of recursive Cantor sets (see e.g. \cite{Ham00} or \cite{Min17}) to $V$-variable Cantor sets:

For $n \in \mathbb{N}_0$ let
\begin{align*}
 K_{n}^{(I)} \coloneqq \bigcup_{\ii \in I_n} S_{\ii}([a,b]).
\end{align*}
A $V$-variable Cantor set $K^{(I)}$ is then given as $K^{(I)} \coloneqq \bigcap_{n=1}^\infty K_n^{(I)}$.

\begin{prop}
The set $K^{(I)}$ is compact and contains at least countably infinitely many elements, namely $S_{(i_1,...,i_n)}(a)$ and $S_{(i_1,...,i_n)} (b)$, $i_1= 1,...,N_{\emptyset},...,i_{n} = 1,...,N_{(i_1,...,i_{n-1})}$.
\end{prop}

\begin{proof}
Let $\ii \in I_n$. For $m \in \mathbb{N}$ let $\ii'$ and $\ii''$ be the two individuals of the population such that $\ii' = \ii  \textbf{1}_m$, $\textbf{1}_m \coloneqq (1,...,1) \in \mathbb{R}^m $ and $i_1'',...,i_{n}'' = i_1,...,i_{n}$, $i''_k = N_{\left(i_1,...,i_{k-1},N_{(i_1,...,i_{k-1})}\right)}$ for $k = n+1,..,n+m$.  By definition, we have
\begin{align*}
S_{\ii'}(a) &= S_{\ii}(a), \\
S_{\ii''}(b) &= S_{\ii}(b).
\end{align*}
Thus, we have $S_{\ii}(a)$, $S_{\ii}(b) \in K_{n+m}^{(I)}$ for all $m \in \mathbb{N}$, which proves the statement.
\end{proof}

Obviously, we have
\begin{align}
K^{(I)} = \bigcup_{i=1}^{N_{\emptyset}} S_i^{(\emptyset)}\left(K^{(\theta_i I)}\right). \label{similarity}
\end{align}

The next step is to construct the $V$-variable Cantor measures, analogously to the homogeneous and recursive Cantor measures. Let 

\begin{align*}
\mu_0(A) &\coloneqq \frac{1}{b-a} \lambda^1_{|_{[a,b]}}(A), \\[8pt]
\mu_n^{(I)}(A) &\coloneqq \sum_{\ii \in I_n} m_{\ii} \, \mu_0\left(S_{\ii}^{-1}(A)\right),  ~~~ n \in \mathbb{N}
\end{align*}

for all $A\in \mathfrak{B}([a,b])$.
The $V$-variable Cantor measure $\mu^{(I)}$ is given as the weak limit of $( \mu^{(I)}_n )_{n \in \mathbb{N}}$. It is easy to see that the weak limit exists and that $\mu^{(I)}$ is a Borel probability measure.

\subsection{Construction of Random $\boldsymbol V$-Variable Cantor Sets and Measures.} \label{random V-variable}
We follow the construction of \cite[Chapter 2.5]{Fre15}. Therefore, let $\mathbb{P}$ be a probability distribution on the index set $J$. From this probability distribution we receive a probability distribution $\mathbb{P}_V$ on the sets of environments by choosing $j_v^E$, $v \in \{1,...,V\}$ independently according to $\mathbb{P}$ and choosing the types $\tau_{v,i}$ $1 \le i \le N_{j_v^E} $ i.i.d. according to the uniform distribution on $\{1,...,V\}$ independently of the chosen indexes. \\\\
Let $\Omega_V$ be the set of $V$-variable trees. We choose $\tau^\emptyset \in \{1,...,V\}$ according to the uniform distribution and independently the environments at each stage i.i.d. according to $\mathbb{P}_V$. This induces a probability distribution on $\Omega_V$ and on the set of $V$-variable fractals $K_V$. For convenience, we denote these probability distributions also by $\mathbb{P}_V$.

\subsection{Necks and Cut Sets.} \label{necks} 
As mentioned in the introduction, an important tool to develop the spectral asymptotics are neck levels which we define in this chapter. Further, we introduce a sequence of cut sets $(\Lambda_k)_k$, related to neck levels. In Chapter \ref{D-N-bracketing} we use this sequence to get a Dirichlet-Neumann-bracketing. A lemma about some asymptotical growth related to individuals in $\Lambda_k$ together with the Dirichlet-Neumann-bracketing will then be used to receive the spectral asymptotics. 
\begin{defi}
Let $E$ be an environment. We call $E$ a \textit{neck} if all $\tau_{v,i}^E$ are equal. Further, we call $n \in \mathbb{N}$ a \textit{neck} of a $V$-variable tree if the environment assigned to the $n$-th generation of the tree $E^n$ is a neck.
\end{defi}
These necks occur with probability one infinitely often and 
\begin{align*}
\mathbb{E}_V n(1) < \infty,
\end{align*}
where we denote by $n(k)$ the $k$-th neck level of the corresponding $V$-variable random tree. Remark that the sequence of times between neck levels is a sequence of geometric random variables. We will need the following property of sums of scale factors,  include from \cite{Fre15}, to determine the spectral exponent.
\begin{lem} 
Let $s_i^{(j)} \in \mathbb{R}$ $i=1,...,N_j$, $j \in J$ such that
\begin{align*}
 s_{\inf} \coloneqq \inf_{j \in J} \min_{i=1,...,N_j} s_i^{(j)} &> 0, \\ s_{\sup} \coloneqq \sup_{j \in J} \max_{i=1,...,N_j} s_i^{(j)} &< \infty.
\end{align*}
Then, with $s_{\ii} \coloneqq s_{i_1}^{(0)} \cdots s_{i_n}^{((i_1,...,i_{n-1}))}$, $\ii = (i_1,...,i_n) \in I$, we have
\begin{align}
\sum_{\ii \in I, \atop |\ii| = n(k)} s_{\ii} = \prod_{j=1}^k \left( \sum_{|\ii| = n(j) - n(j-1)} s_{\ii_{|_j}} \right), \label{sum scale decomp}\\[5pt]
\lim_{k \rightarrow \infty} \frac{1}{k} \log \sum_{|\ii| = n(k)} s_{\ii} = \mathbb{E}_V \log \sum_{|\ii| = n(1)} s_{\ii}, \label{sum log scale asymptotics}~~~ a.s.
\end{align} 
\end{lem}
Next we define cut sets and the sequence of cut sets considered in this work.

\begin{defi}
Let $\partial I$ be the set of infinite paths through $I$, beginning at 0. A set $\Lambda \subset I$ is called a \textit{cut set} of the tree $I$ if for every $\omega \in \partial I$ there exists exactly one $\ii \in \Lambda$ such that $\left(\omega_1,...,\omega_{|\ii|}\right) = \ii$, where $|\ii|$ is the length of the vector $\ii$. 
\end{defi}
The sequence of cut sets we are interested in is given by
\begin{align*}
\Lambda_0 &\coloneqq {\emptyset}, \\
\Lambda_k &\coloneqq \left\{\ii \in I: \exists l \in \mathbb{N} \, : |\ii| = n(l) \, : ~~ m_{\ii} \, r_{\ii} \le e^{-k} < m_{\ii_{|_{n(l-1)}}} \, r_{\ii_{|_{n(l-1)}}}      \right\},
\end{align*}
where $\ii_{|_k} \coloneqq (i_1,...,i_k)$ for $k \le |\ii|$. Next, we compare the asymptotics of objects, related to these cut sets. Therefore, we use the following notation. Let $f,g$ be real valued functions. We say $f$ is asymptotically dominated by $g$ and write
\begin{align*}
f \preceq g ~~~ \text{ iff } ~~~ \limsup_{k \rightarrow \infty} \frac{f(k)}{g(k)} \le 1.
\end{align*}
Then, let
\begin{align*}
M_k \coloneqq |\Lambda_k|,& ~~~ T_k \coloneqq \frac{M_k}{\sum\limits_{\ii \in \Lambda_k} r_{\ii}\, m_{\ii}}, \\
y_k(\ii) \coloneqq n(l) - n(l-1),& \text{ for } \ii \in \Lambda_k, |\ii| = n(l), ~~~ y_k \coloneqq \sup_{\ii \in \Lambda_k} y_k(\ii).
\end{align*}
The asymptotics we give are slight modifications of \cite[Lemma 3.8.(c)]{Fre15}.
\begin{lem} \label{estimate rinf minf}
There exists $\alpha'>0$ such that 
\begin{align*}
k^{- \alpha'} e^{-k} \preceq (r_{\inf} \, m_{\inf})^{y_k} \, e^{- k} \le r_{\ii} \, m_{\ii} \le e^{-k}, ~~~ a.s. ~~ \text{ for } \ii \in \Lambda_k.
\end{align*}

\end{lem}

\section[5pt]{Spectral Asymptotics for $V$-Variable Cantor Measures} \label{section asympotic}
 
\subsection{Preliminaries.}
To receive the Dirichlet-Neumann-bracketing under consideration, we need some scaling properties of the eigenvalue counting functions. We prepare this by giving the scaling properties for the $L_2(\mu^{(I)})$-Norm of $L_2(\mu^{(I)})$ functions. This scaling property is a corollary of the $L_2$-Norm scaling property given in \cite{Min17}.

\begin{lem}
For all $\ii \in I$ holds
\begin{align*}
\mu^{(I)} (S_{\ii}([a,b])) = m_{\ii}.
\end{align*}
\end{lem}

\begin{proof}
We write $\mu = \mu^{(I)}$, $\mu_n = \mu_n^{(I)}$. Let $K_{\ii} \coloneqq S_{\ii}([a,b])$ for $\ii \in I$. Let $\ii \in I_{n}$, $\jj \in I_{n+m}$. 
Because of 
\begin{align*}
K_{\ii} \cap K_{\jj}= 
\begin{cases}
K_{\jj},& \text{ if } \jj|_n = \ii, \\
\emptyset,& \text{ otherwise},
\end{cases}
\end{align*}
we get
\begin{align*}
&\mu_{n+m}(K_{\ii}) \\[5pt]
=&\sum_{\jj \in I_{n+m}}m_{\jj} \, \mu_0\left(S_{\jj}^{-1}(K_{\ii})\right) \\[5pt]
=&\sum_{\jj \in I_{n+m} \atop \jj|_n = \ii } m_{\jj} \, \mu_0\left(S_{\jj}^{-1}(K_{\ii})\right).
\end{align*} 

Because of
\begin{align*}
&\left(S_{(i_1,...,i_n,j_{n+1},...,j_{n+m})})^{-1}(K_{\ii}\right) \\[5pt] = 
&\left(S_{j_{n+1}}^{((i_1,...,i_n))} \circ \cdots \circ S_{j_{n+m}}^{((i_1,...,i_n,j_{n+1},...,j_{n+m-1}))}\right)^{-1} \circ S_{\ii}^{-1}   (K_{\ii}) \\[5pt] = 
&\left(S_{j_{n+1}}^{((i_1,...,i_n))} \circ \cdots \circ S_{j_{n+m}}^{((i_1,...,i_n,j_{n+1},...,j_{n+m-1}))}\right)^{-1} ([a,b]) \\[5pt] =& [a,b],
\end{align*}

we get 
\begin{align*}
\mu_{n+m}(K_{\ii}) &= \sum_{\jj \in I_{n+m} \atop \jj|_n = \ii } m_{\jj} = m_{\ii}.
\end{align*}

\end{proof}
Analogously to \eqref{similarity}, it holds 
\begin{align}
\mu^{(I)} = \sum_{i=1}^{N_{\emptyset}} m_i^{(\emptyset)} \, S_i^{(\emptyset)} \mu^{(\theta_i I)}, \label{measure similarity}
\end{align}
where $S_i^{(\emptyset)} \mu^{(I)}(A) \coloneqq \mu^{(I)}\left(\left(S_i^{(\emptyset)}\right)^{-1}(A)\right)$, $A \in \mathcal{B}([a,b])$. 
\begin{proof}
Let $A \in \mathcal{B}([a,b])$. Then, 
\begin{align*}
&\sum_{i=1}^{N_{\emptyset}} m_i^{(\emptyset)} \, \mu_n^{(\theta_i I)}\left(\left(S_i^{(\emptyset)}\right)^{-1}(A)\right) \\   
=& \sum_{i=1}^{N_{\emptyset}} \sum_{i_1=1}^{N_{i}} \cdots \sum_{i_n=1}^{N_{(i,i_1,...,i_{n-1})}} m_{i}^{(\emptyset)} \, m_{i_1}^{(i)} \cdots m_{i_{n}}^{((i,i_1,...,i_{n-1}))} \, \mu_\emptyset\left(\left(S_{i,i_1,...,i_{n}}\right)^{-1}(A)\right) \\
=& \sum_{i_1=1}^{N_{\emptyset}} \cdots \sum_{i_{n+1}=1}^{N_{(i_1,...,i_{n})}} m_{i_1}^{(\emptyset)} \, m_{i_2}^{(i_1)} \cdots m_{i_{n+1}}^{((i_1,...,i_{n}))} \, \mu_\emptyset\left(\left(S_{i_1,...,i_{n+1}}\right)^{-1}(A)\right) \\
=& \mu_{n+1}^{(I)}(A).
\end{align*}
Taking the limit $n \rightarrow \infty$, we get the assertion.
\end{proof}
With \eqref{measure similarity} we get the following lemma.
\begin{prop} \label{scale measure}
Let $i \in \{1,...,N_{\emptyset}\}$ and $A \in \mathcal{B}([a,b])$ with $A \subseteq S_i^{(I)}([a,b])$. Then, it holds
\begin{align*}
\mu^{(I)}(A) = m_i^{(\emptyset)} \, \left(S_i^{(\emptyset)} \mu^{(\theta_i I)}\right)(A).
\end{align*}
\end{prop}

\begin{lem} \label{scaleprop E}
Let $f,g \in L_2\left(\mu^{(I)}\right)$. Then, 
\begin{align*}
\left\langle f,g \right\rangle_{L_2\left(\mu^{(I)}\right)} = \sum_{i=1}^{N_{\emptyset}} m_i^{(\emptyset)} \left\langle f \circ S_i^{(\emptyset)},g\circ S_i^{(\emptyset)} \right\rangle_{L_2\left(\mu^{\left(\theta_i I\right)}\right)}.
\end{align*}
\end{lem}
\begin{proof}
With $\supp \mu^{(I)} = K^{(I)}$ and Proposition \ref{scale measure}, we get 

\begin{align*}
\langle f,g \rangle_ {L_2(\mu^{(I)})} &= \int_{[a,b]} f \, g \, d \mu^{(I)} \\
&= \sum_{i=1}^{N_{\emptyset}} \int_{S_i^{(\emptyset)}([a,b])} f\, g \, d \mu^{(I)}  \\
&= \sum_{i=1}^{N_{\emptyset}} \int_{[a,b]} f \circ S_i^{(\emptyset)} \, g \circ S_i^{(\emptyset)} \, d\left(S_i^{(\emptyset)^{-1}} \mu^{(I)}\right) \\
&= \sum_{i=1}^{N_{\emptyset}} m_i^{(\emptyset)} \int_{[a,b]} f  \circ S_i^{(\emptyset)} \, g \circ S_i^{(\emptyset)} \, d \mu^{\left(\theta_i I\right)} \\
&= \sum_{i=1}^{N_{\emptyset}} m_i^{(\emptyset)} \langle f\circ S_i^{(\emptyset)},g\circ S_i^{(\emptyset)} \rangle_{L_2\left(\mu^{(\theta_i I)}\right)}.
\end{align*}
\end{proof}

Iteratively, we receive:
\begin{prop}\label{scaleprop E iteratively}
Let $\Lambda \subset I$ be a cut set of $I$. Then, it holds
\begin{align*}
\left\langle f,g \right\rangle_{L_2\left(\mu^{(I)}\right)} = \sum_{\ii \in \Lambda} m_{\ii} \left\langle f \circ S_{\ii},g\circ S_{\ii} \right\rangle_{L_2\left(\mu^{\left(\theta_{\ii} I\right)}\right)}
\end{align*}
\end{prop}

\subsection{Dirichlet-Neumann-Bracketing.} \label{D-N-bracketing}
We begin by giving the scaling property for the Neumann eigenvalue counting function.
Therefore, let $(\mathcal{E},\mathcal{F})=(\mathcal{E}^{(I)},\mathcal{F})$ be  the Dirichletform on $L_2(\mu^{(I)})$, whose eigenvalues coincide with the Neumann eigenvalues of $\frac{d}{d \mu^{(I)}} \frac{d}{d x}$. Namely, 
\begin{align*}
\mathcal{F} &= H^1(\lambda), \\
\mathcal{E}(f,g) &= \int_a^b f'(x) \, g'(x) \, dx,
\end{align*}
see \cite[Proposition 5.1]{Fre04/05}.
We write $N^{(I)}_N$ for the eigenvalue counting function of $(\mathcal{E},\mathcal{F})$, instead of $N_{(\mathcal{E},\mathcal{F})}$.
To obtain the spectral asymptotics, we will estimate the eigenvalue counting functions. Therefore, we will need a sequence of Dirichlet-Neumann-Bracketings, depending on $\Lambda_k$ defined in Chapter \ref{necks}. Since $\Lambda_k$ is for all $k \in \mathbb{N}$ a cut set, there exists an $n  \in \mathbb{N}$ such that $N_k \coloneqq \left(N_\emptyset,N_{(N_\emptyset)},N_{\left(N_\emptyset,N_{(N_\emptyset)}\right)},...\right)$, $|N_k| = n$ and $N_k \in \Lambda_k$. To each $\ii \in \Lambda_k \backslash \{N_k\}$ there exists a $\ii' \in \Lambda$ such that the left neighbour point in $K^{(I)}$ of $S_{\ii}(b)$ is $S_{\ii'}(a)$. Then, we define the gap interval between $S_{\ii}[a,b]$ and $S_{\ii'}[a,b]$ by $I_{\ii}$, i.e. $I_{\ii} \coloneqq (S_{\ii}(b),S_{\ii'}(a))$. \\\\
For the bracketing, we define a sequence of Dirichlet forms $\left((\mathcal{E}^k,\mathcal{F}^k)\right)_{k \in \mathbb{N}}$. Therefore, let
\begin{align*}
\mathcal{F}^k \coloneqq \left\{f:[a,b] \longrightarrow \mathbb{R
} \,: f \circ S_{\ii} \in H^1\left(\lambda^1_{|_{[a,b]}}\right) ~ \forall \ii \in \Lambda_k ~ \text{ and } ~ f_{|_{I_{\ii}}} \in H^1\left(\lambda^1_{|_{I_{\ii}}}\right)\right\}.
\end{align*}
By using \cite[Proposition 3.2.1]{Arz14} iteratively, we receive:
\begin{prop}\label{Dirichlet scaling}
Let $f,g \in \mathcal{F}$ and $k \in \mathbb{N}$. Then, for all $\ii \in \Lambda_k$, $f\circ S_{\ii}, g \circ S_{\ii} \in \mathcal{F}$ and
\begin{align*}
\mathcal{E}(f,g) = \sum_{\ii \in \Lambda_k} \frac{1}{r_{\ii}} \mathcal{E}\left(f \circ S_{\ii}, g \circ S_{\ii}\right) + \sum_{\ii \in \Lambda_k\backslash \{N_k\}} \int_{I_{\ii}} f'(t) \, g'(t) \, dt.
\end{align*}
\end{prop} 
Therefore, if we define 
\begin{align*}
\mathcal{E}^k(f,g) \coloneqq \sum_{i \in \Lambda_k} \frac{1}{r_{\ii}} \mathcal{E}\left(f \circ S_{\ii}, g \circ S_{\ii}\right) + \sum_{\ii \in \Lambda_k\backslash \{N_k\}} \int_{I_{\ii}} f'(t) \, g'(t) \, dt, ~~~ f,g \in \mathcal{F}^k,
\end{align*}
we have $(\mathcal{E},\mathcal{F}) \subseteq (\mathcal{E}^k,\mathcal{F}^k)$. As in \cite[Chapter 3.2.2]{Arz14} we receive that $(\mathcal{E}^k, \mathcal{F}^k)$ is a Dirichlet form on $L_2(\mu^{(I)})$ and that the embedding $\mathcal{F}^k \hookrightarrow L_2(\mu^{(I)})$ is a compact operator. Thus, we can refer to the eigenvalue counting function $N^k_N$ of  $(\mathcal{E}^k,\mathcal{F}^k)$.

\begin{prop}
For all $x \ge 0$, $k \in \mathbb{N}$  holds 
\begin{align*}
N_N^k(x) = \sum_{\ii \in \Lambda_k} N^{(\theta_{\ii} I)}_N\left(r_{\ii} \, m_{\ii} \, x\right).
\end{align*}
\end{prop}
\begin{proof}
Let $f$ be an eigenfunction of $\left(\mathcal{E}^k, \mathcal{F}^k,\mu^{(I)}\right)$ with eigenvalue $\lambda$, i.e. 
\begin{align*}
\mathcal{E}^k (f,g) = \lambda \, \langle f , g \rangle_{L_2\left(\mu^{(I)}\right)} ~~~ \text{ for all } g \in \tilde{\mathcal{F}}.
\end{align*}
Because $f,g \in L_2\left(\mu^{(I)}\right)$, we have with Proposition \ref{scaleprop E iteratively} 
\begin{align}
\begin{split}
&\sum_{\ii \in \Lambda_k} \frac{1}{r_{\ii}} \mathcal{E}\left(f \circ S_{\ii}, g \circ S_{\ii}\right) + \sum_{\ii \in \Lambda_k\backslash \{N_k\}} \int_{I_{\ii}} f'(t) \, g'(t) \, dt \\  &= \lambda \, \sum_{\ii \in \Lambda_k} m_{\ii} \, \left\langle f \circ S_{\ii},g \circ S_{\ii} \right\rangle_{L_2\left(\mu^{\left(\theta_{\ii} I\right)}\right)}. \label{efeq}
\end{split}
\end{align}
Now, we show that each summand in the first sum on the left side equals each summand on the right side, respectively. Therefore, let $h \in \mathcal{F}$ and define for each $\jj \in \Lambda_k$
\begin{align*}
h^k_{\jj}(x) \coloneqq
\begin{cases}
h \circ S_{\jj}(x), &\text{ if } x \in S_{\jj}([a,b]), \\
0, &\text{otherwise.}
\end{cases}
\end{align*}
Obviously, we have $h^k_{\jj} \in \mathcal{F}^k, h^k_{\jj} \circ S_{\jj} = h$ for all $\jj \in \Lambda_k$ and $h^k_{\jj} \circ S_{\ii} = 0$ for $\Lambda_k \ni \ii \neq \jj$. Moreover, $h_{\jj}'\big|_{I_{\ii}} = 0$, for all $\jj \in \Lambda_k$, $\ii \in \Lambda_k \backslash \{N_k\}$. With $g = h_{\jj}$, we then have in \eqref{efeq}  
\begin{align*}
 \frac{1}{r_{\jj}} \, \mathcal{E}\left(f \circ S_{\jj}, h\right) = \lambda \, m_{\jj} \, \left\langle f \circ S_{\jj},h \right\rangle_{L_2\left(\mu^{(\theta_{\jj}  I)}\right)}.
\end{align*}
Because this equation holds for all $h \in \mathcal{F}$, $f \circ S_{\jj}$ is an eigenfunction of the Dirichlet form $\left(\mathcal{E}, \mathcal{F},\mu^{\left(\theta_{\jj}I\right)}\right)$ with eigenvalue  $r_{\jj} \, m_{\jj} \, \lambda$ for all $\jj \in \Lambda_k$.
\\\\
Now, let $\lambda > 0$ such that for $\ii \in \Lambda_k$ $r_{\ii} \, m_{\ii} \, \lambda$ is an eigenvalue of $\left(\mathcal{E}, \mathcal{F},\mu^{(\theta_{\ii}I)}\right)$ with eigenfunction $f_{\ii}$, say. This means, 
\begin{align*}
\mathcal{E}(f_{\ii},g) = r_{\ii} \, m_{\ii} \, \lambda \, \left\langle f_{\ii}, g \right\rangle_{L_2\left(\mu^{(\theta_{\ii} I)}\right)}
\end{align*}
for all $g \in \mathcal{F}$. Let
\begin{align*}
f(x) \coloneqq 
\begin{cases}
f_{\ii} \circ S_{\ii}^{-1}(x), &\text{if } x \in S_{\ii}([a,b]) \text{ for some } \ii \in \Lambda_k, \\
0, &\text{otherwise}.
\end{cases}
\end{align*}
Then $f \in \mathcal{F}^k$ and $f \circ S_{\ii} = f_{\ii}$, $\ii \in \Lambda_k$ and therefore
\begin{align*}
\sum_{\ii \in \Lambda_k} \frac{1}{r_{\ii}} \mathcal{E}\left(f \circ S_{\ii}, g \right) = \lambda \, \sum_{\ii \in \Lambda_k} m_{\ii} \, \left\langle f \circ S_{\ii},g  \right\rangle_{L_2\left(\mu^{\left(\theta_{\ii} I\right)}\right)}
\end{align*}
for all $g \in \mathcal{F}$. Since for $g_k \in \mathcal{F}^k$ we have by definition of $\mathcal{F}$, $g_k \circ S_{\ii} \in\mathcal{F}$, $\ii \in \Lambda_k$, we get
\begin{align*}
\sum_{\ii \in \Lambda_k} \frac{1}{r_{\ii}} \mathcal{E}\left(f \circ S_{\ii}, g_k \circ S_{\ii}\right) = \lambda \, \sum_{\ii \in \Lambda_k} m_{\ii} \, \left\langle f \circ S_{\ii},g_k \circ S_{\ii} \right\rangle_{L_2\left(\mu^{(\theta_{\ii} I)}\right)}. 
\end{align*} 
But the left side of this equation is equal to $\mathcal{E}^k(f,g_k)$, because $f'\big|_{I_{\ii}} = 0$ for all $\ii \in \Lambda_k \backslash \{N_k\}$. With Proposition \ref{scaleprop E iteratively} we then have
\begin{align*}
\mathcal{E}^k (f,g_k) = \lambda  \, \langle f, g_k \rangle_{L_2\left(\mu^{(I)}\right)}
\end{align*}
for all $g_k \in \mathcal{F}^k$. Therefore, $\lambda$ is an eigenvalue of $\left(\mathcal{E}^k , \mathcal{F}^k , \mu^{(I)}\right)$ with corresponding eigenfunction $f$. Using this, we can easily conclude the claim. 
\end{proof}

Next, we give the scaling property of the Dirichlet eigenvalue counting function.
Therefore, let $(\mathcal{F}_0,\mathcal{E})$ be the Dirichletform on $L_2\left(\mu^{(I)}\right)$ whose eigenvalues coincide with the Dirichlet eigenvalues of $\frac{d}{d \mu^{(I)}} \frac{d}{d x}$. Meaning, $\mathcal{E}$ is defined as before and 
\begin{align*}
\mathcal{F}_0 \coloneqq \{f \in \mathcal{F}: ~ f(a)=f(b)=0\}.
\end{align*}
Again, we define a sequence of Dirichlet forms $\left(\mathcal{E},\mathcal{F}_0^k\right)$ on $L_2\left(\mu^{(I)}\right)$, where
\begin{align*}
\mathcal{F}^{k}_0 \coloneqq \{f \in \mathcal{F}_0 : ~ f(x) = 0 \text{ for } x \in I_{\ii},~ \ii \in \Lambda_k \backslash \{N_k\}\}, ~~~ k\in \mathbb{N}.
\end{align*}
Further, we use the notation $\mathcal{E}$ for $\mathcal{E}\big|_{\mathcal{F}^k_0 \times \mathcal{F}^k_0}$ and denote the corresponding eigenvalue counting function by $N_D^k$.

\begin{prop}
For all $x \ge 0$ we have
\begin{align*}
N_D^k(x) = \sum_{\ii \in \Lambda_k} N_D^{(\theta_{\ii} I)} \left(r_{\ii} \, m_{\ii} \, x \right).
\end{align*}
\end{prop}

\begin{proof}
Let $f$ be an eigenfunction of $\left(\mathcal{E},\mathcal{F}_0^k,\mu^{(I)}\right)$ with eigenvalue $\lambda$. Then,
\begin{align*}
\mathcal{E}(f,g) = \lambda \, \langle f,g \rangle_{L_2\left(\mu^{(I)}\right)},
\end{align*}
for all $g \in \mathcal{F}_0^k$.  Therefore, we have with Proposition \ref{Dirichlet scaling} and Lemma \ref{scaleprop E iteratively},
\begin{align*}
&\sum_{\ii \in \Lambda_k} \frac{1}{r_{\ii}} \mathcal{E}\left(f \circ S_{\ii}, g \circ S_{\ii}\right) + \sum_{\ii \in \Lambda_k\backslash \{N_k\}} \int_{I_{\ii}} f'(t) \, g'(t) \, dt \\  &= \lambda \, \sum_{\ii \in \Lambda_k} m_{\ii} \, \left\langle f \circ S_{\ii},g \circ S_{\ii} \right\rangle_{L_2\left(\mu^{\left(\theta_{\ii} I\right)}\right)}.
\end{align*}
For $h \in \mathcal{F}_0$ we define
\begin{align*}
h_{\jj}^k(x) \coloneqq
\begin{cases}
h \circ S_{\jj}^{-1}(x), &\text{if } x \in S_{\jj}([a,b]), \\
0, &\text{otherwise}.
\end{cases}
\end{align*}
Because $h \in \mathcal{F}_0$, it follows $h_{\jj}^k \in \mathcal{F}_0^k$ and $h_{\jj} \circ S_j^{(0)} = h$ for $\jj \in \Lambda_k$ and $h_{\jj} \circ S_{\ii} = 0$ if $\Lambda_k \ni \ii \neq \jj$. Hence,
\begin{align*}
\frac{1}{r_{\jj}} \, \mathcal{E}\left(f \circ S_{\jj}, h\right) = \lambda \, m_{\jj} \, \left\langle f \circ S_{\jj},h \right\rangle_{L_2\left(\mu^{\left(\theta_{\jj} I\right)}\right)},
\end{align*}
for all $\jj \in \Lambda_k$. Therefore, $\lambda \, r_{\jj} \, m_{\jj}$ is an eigenvalue of $\left(\mathcal{E},\mathcal{F}_0, \mu^{\left(\theta_{\jj} I\right)}\right)$ with eigenfunction $f \circ S_{\jj}$, $\jj \in \Lambda_k$.
\\\\
Now, let $r_{\ii} \,m_{\ii}\, \lambda$ be an eigenvalue of $\left(\mathcal{E},\mathcal{F}_0,\mu^{\left(\theta_{\ii} I\right)}\right)$ for some $\lambda >0$ with corresponding eigenfunction $f_{\ii}$, $\ii \in \Lambda_k$. Therefore, we have
\begin{align*}
\mathcal{E} (f_{\ii},g) = r_{\ii} \,  m_{\ii} \,  \lambda \, \langle f_{\ii},g  \rangle_{L_2\left(\mu^{\left(\theta_{\ii} I\right)}\right)} 
\end{align*}
for all $g \in \mathcal{F}_0$. Let 
\begin{align*}
f(x) \coloneqq 
\begin{cases}
f_{\ii} \circ S_{\ii}^{-1}(x), &\text{if } x \in S_{\ii}([a,b]) \text{ for some } \ii \in \Lambda_k, \\
0, &\text{otherwise}.
\end{cases}
\end{align*}
Since $f_{\ii} \in \mathcal{F}_0$, we have $f \in \mathcal{F}_0^k$ and because of $ f \circ S_{\ii} = f_{\ii}$, $\ii \in \Lambda_k$, we have
\begin{align*}
\sum_{\ii \in \Lambda_k} \frac{1}{r_{\ii}} \mathcal{E}\left(f \circ S_{\ii}, g \right) = \lambda \, \sum_{\ii \in \Lambda_k} m_{\ii} \, \left\langle f \circ S_{\ii},g  \right\rangle_{L_2\left(\mu^{\left(\theta_{\ii} I\right)}\right)}
\end{align*}
for all $g \in \mathcal{F}_0$. For $g_k\in \mathcal{F}_0^k$ we have $g_k \circ S_{\ii} \in \mathcal{F}_0$, $\ii \in \Lambda_k$. Analogously to the case with Neumann boundary conditions, we get
with Proposition \ref{Dirichlet scaling} and Proposition \ref{scaleprop E iteratively},
\begin{align*}
\mathcal{E} (f,g_k) = \lambda  \, \langle f, g_k \rangle_{L_2\left(\mu^{(I)}\right)}
\end{align*}
for all $g_k \in \mathcal{F}_k$. Hence, $\lambda$ is an eigenvalue of $\left(\mathcal{E},\mathcal{F}^k_0, \mu^{(I)}\right)$ with eigenfunction $f$ and, as before, we can now easily conclude the claim.
\end{proof}

Since $\left(\mathcal{E}^k, \mathcal{F}^k, \mu^{(I)}\right)$ is an extension of $\left(\mathcal{E}, \mathcal{F}, \mu^{(I)}\right)$ and $\left(\mathcal{E}, \mathcal{F}_0, \mu^{(I)}\right)$ is an extension of $(\mathcal{E},\mathcal{F}^k_0, \mu^{(I)})$ for all $k \in \mathbb{N}$, we finally receive the needed Dirichlet-Neumann-Bracketing:

\begin{kor}[Dirichlet-Neumann-Bracketing] \label{counting scale}
For all $x \ge 0$ and $k \in \mathbb{N}$ holds
\begin{align*}
\sum_{\ii \in \Lambda_k} N_D^{\left(\theta_{\ii} I\right)} \left(r_{\ii} \, m_{\ii} \, x\right) \le N_D^{(I)}(x) \le  N^{(I)}_N(x)  \le \sum_{\ii \in \Lambda_k} N^{\left(\theta_{\ii} I\right)}_N\left(r_{\ii}\, m_{\ii}\,x\right).
\end{align*}
\end{kor}

\subsection{Eigenvalue Estimates.}
In this Chapter we give estimates for the Dirichlet eigenvalues. As before, we fix a $V$-variable tree $I$. We write $\lambda_{D,1}$ for the first Dirichlet eigenvalue of $- \frac{d}{d \mu^{(I)}} \frac{d}{d x}$ and $\mu$ for $\mu^{(I)}$.
\begin{lem} \label{first D estimate}
It holds 
\begin{align*}
\frac{1}{(b-a)} \le \lambda_{D,1} \le \frac{1 - r_{\inf}^2}{\left(r_{\inf} \, m_{\inf} (1- r_{\sup})\right)^2 (b-a)}.
\end{align*}
\end{lem}
\begin{proof}
For the first estimate, let $f$ be an eigenfunction of $(\mathcal{E},\mathcal{F}_0,\mu)$ such that $\lVert f \rVert_{L_2(\mu)} = 1$. By the Cauchy-Schwarz inequality, we receive
\begin{align*}
f^2(x) = (f(x) - f(a))^2 = \left( \int_a^x f'(y) dy \right)^2 \le \lVert f'\rVert_{L_2(\lambda^1,[a,x])}^2 \, (x-a) \le \lVert f'\rVert_{L_2(\lambda^1,[a,b])}^2 \, (b-a). 
\end{align*}
Integrating with respect to $\mu$ yields
\begin{align*}
1 \le \lVert f'\rVert_{L_2(\lambda^1,[a,b])}^2 \, (b-a).
\end{align*}
Since $f$ is an eigenfunction of $\mathcal{E}$, we have 
\begin{align*}
\lVert f' \rVert_{L_2(\lambda^1,[a,b])}^2 = \langle f',f' \rangle_{L_2(\lambda^1, [a,b])} = \mathcal{E}(f,f) = \lambda_{D,1}.
\end{align*}
Hence, the first estimate follows. For the second estimate, define $x_1 \coloneqq S_1^{(\emptyset)}  \left(S_{N_{(1)}}^{(1)} (a)\right)= a+r_1^{(\emptyset)}\left( 1-r_{N_{(1)}} \right)(b-a)$, $x_2 \coloneqq S_1^{(\emptyset)}(b) = a+ r_1^{(\emptyset)}(b-a)$ and
\begin{align*}
\hat{f}(x) \coloneqq
\begin{cases}
\frac{x-a}{x_1-a}, ~~~ &\text{if } x \in [a,x_1] \\
1, &\text{if } x \in (x_1,x_2] \\
\frac{b-x}{b-x_2}, ~~~ &\text{if } x \in (x_2,b].
\end{cases}
\end{align*}
Therefore, $\hat{f}$ is constant 1 on the very right second-level cell which remains from the very left first-level cell and linear interpolated from $a$ to $x_1$ and $b$ to $x_2$ such that $\hat{f} \in \mathcal{F}_0$. 
 Hence,
\begin{align*}
\mathcal{E}(\hat{f},\hat{f}) &= \int_a^b \left(\hat{f}'\right)^2 \, dx \\
&= \frac{1}{x_1-a} + \frac{1}{b-x_2} \\
&= \frac{1-r_1^{(\emptyset)} + r_1^{(\emptyset)}\left(1-r_{N_{(1)}}^{(1)} \right) }{ r_1^{(\emptyset)} \left(1-r_{N_{(1)}}^{(1)} \right) \left(1-r_1^{(\emptyset)}\right)(b-a) }. 
\end{align*}
Further, we have 
\begin{align*}
\int_a^b \left(\hat{f}\right)^2 d \mu \ge m_1^{(\emptyset)} \, m_{N_{(1)}}^{(1)}.
\end{align*}
Together with Rayleigh, we receive
\begin{align*}
\lambda_{D,1} = \inf_{f \in \mathcal{F}_0} \frac{\mathcal{E}(f,f)}{\lVert f \rVert_{L_2(\mu)}^2} &\le \frac{\mathcal{E}(\hat{f},\hat{f})}{\left\lVert \hat{f} \right\rVert_{L_2(\mu)}^2} \\
&\le \frac{1-r_1^{(0)} + r_1^{(0)}\left(1-r_{N_{(1)}}^{(1)} \right) }{ r_1^{(0)} \left(1-r_{N_{(1)}}^{(1)} \right) \left(1-r_1^{(0)}\right)(b-a) \, m_1^{(0)} \, m_{N_{(1)}}^{(1)}} \\
&\le \frac{1 - r_{\inf}^2}{\left(r_{\inf} \, m_{\inf} (1- r_{\sup})\right)^2 (b-a)}.
\end{align*}
\end{proof}

\begin{lem} 
Let $\tau$ be a finite non-atomic Borel measure on $[a,b]$ with $a,b \in \supp \tau$. Then, there exists $c >0$ such that
\begin{align*}
N^{\tau}_D(x) \le \tau([a,b]) \,c\, x ~~~ a.s.
\end{align*}
Moreover, $c$ is independent of $\tau$.
\end{lem}
\begin{proof}
Let 
\begin{align*}
g(x,y) \coloneqq \frac{\min(x-a,y-a) \min(b-y,b-x)}{b-a}.
\end{align*}
Then, with
\begin{align*}
T_g : L_2(\tau) &\longrightarrow L_2(\tau) \\
f &\mapsto \int_a^b g(\cdot,y) f(y) \, d\tau(y),
\end{align*}
we have
\begin{align*}\begin{cases}
-\frac{d}{d \tau} \frac{d}{d x} f = \lambda f \\[5pt]
f(a) = f(b) = 0
\end{cases}
~~~ \text{ iff } ~~~ T_gf= \frac{1}{ \lambda} f,
\end{align*}
cf. \cite[Theorem 4.1]{Fre03}. By \cite[Definition 4.1, Lemma 4.3, Lemma 4.6]{Ste08}, $g$ is a continuous kernel and thus, we can use Mercer's Theorem \cite[Theorem 4.49]{Ste08} and therefore
\begin{align*}
g(x,y) = \sum_{i=1}^\infty \frac{1}{\lambda_{D,i}^\tau} \, f_i(x) \, f_i(y),
\end{align*}
where $f_i$ is a normalized eigenfunction to the eigenvalue $\lambda_{D,i}^\tau$. Furthermore, the convergence is uniform. Since $g$ is bounded, there exists a $c>0$ such that
\begin{align*}
c \ge \sum_{i=1}^\infty \frac{1}{\lambda_{D,i}^\tau} \, f_i(x) \, f_i(x).
\end{align*}
Integrating both sides with respect to $\tau$, we receive
\begin{align*}
\tau([a,b]) \, c \ge \sum_{i=1}^\infty \frac{1}{\lambda_{D,i}^\tau} = \int_0^\infty \frac{1}{s} \, dN_D^{\tau}(s) \ge \int_0^x \frac{1}{s} \, dN_D^{\tau}(s) \ge \frac{1}{x} \, N_D^{\tau}(x)
\end{align*} 
and thus the claim follows.
\end{proof}
With this lemma, we can estimate $N_D^{(I)}$.
\begin{kor} \label{DECF smaler x}
There exists a $c>0$ independent of $I$ such that for all $x>0$ 
\begin{align*}
N_D^{(I)}(x) \le c \, x.
\end{align*}
\end{kor}

In the following, let $\eta \coloneqq r_{\inf} \, m_{\inf}$.

\begin{lem} \label{estimate DECF 2}
There exists $c_1,c_2 > 0$  such that for almost all $\omega \in \Omega$
\begin{align*}
N_D^{(I)}(T_k) \le c_1 M_k, ~~~ M_k \le N_D^{(I)} (c_2 T_k \eta^{- y_k}) 
\end{align*}
for all $k \ge 0$
\end{lem}

\begin{proof}
With the Dirichlet-Neumann-Bracketing Lemma \ref{counting scale} and Corollary \ref{DECF smaler x}, we have
\begin{align*}
N_D^{(I)}(T_k) &\le \sum_{\ii \in \Lambda_k} N^{(\theta_{\ii}I)}_N (m_{\ii} r_{\ii} T_k) \\
&\le 2\, M_k + \sum_{\ii \in \Lambda_k} N_D^{(\theta_{\ii}I)}(m_{\ii} r_{\ii} T_k) \\
&\le 2 \, M_k + c \, T_k\sum_{\ii \in \Lambda_k} m_{\ii} r_{\ii} \\
&\le c_1 M_k.
\end{align*}
Where we used $N_N^{(I)}(x) \le N_D^{(I)}(x) + 2$ (see \cite[Proposition 5.5]{Fre04/05}) for the second inequality. \\
For the second estimate remark that $\lambda_{D,1} < \lambda_{D,2}$. Together with $(r_{\ii} \, m_{\ii} )^{-1} \le \eta^{-y_k} e^{-k} \le \eta^{-y_k} T_k$, $ \ii \in \Lambda_k$ for all $k$, which follows from Lemma \ref{estimate rinf minf}, and Lemma \ref{counting scale}, Lemma \ref{first D estimate}, we get
\begin{align*}
M_k = \sum_{\ii \in \Lambda_k} N_D^{(\theta_{\ii}I)}\left(\lambda_{D,1}^{\mu^{(\theta_{\ii}I)}} \right) \le \sum_{\ii \in \Lambda_k} N_D^{(\theta_{\ii}I)}\left(c_2 r_{\ii} \, m_{\ii} \, (r_{\ii} \, m_{\ii} )^{-1} \right) \le N_D^{(I)} (c_2 T_k \eta^{-y_k}).
\end{align*}
\end{proof}

\begin{lem} \label{estimate DECF 3}
 $\mathbb{P}$-a.s. there exists $k_0(\omega) \in \mathbb{N}$ and $\alpha, c_1 > 0$ such that
\begin{align*}
N_D^{(I)}(T_k) \le c_1 M_k, ~~~ M_k \le N^{(I)}_D(k^\alpha T_k), ~~~ \text{ for } k > k_0(\omega).
\end{align*}
\begin{proof}
The lemma follows from Lemma \ref{estimate DECF 2} and $\eta^{-y_k} \preceq k^{\alpha'}$ by Lemma \ref{estimate rinf minf}.
\end{proof}

\end{lem}

\subsection{Spectral Exponent.} In this Chapter the spectral exponent is calculated. We will see that the spectral exponent is given as the unique zero strictly bigger than zero of the function defined in the next lemma. This lemma shows that this zero is indeed unique and exists. The proof is a slight modification of the proof of \cite[Lemma 4.12]{Fre15}.
\begin{lem} \label{existence}
Let
\begin{align*}
f(x) \coloneqq \mathbb{E}_V \log \sum_{|\ii| = n(1)} \left( m_{\ii} r_{\ii} \right)^{x}, ~~~ x > 0.
\end{align*}
Then, there exists a unique $\gamma > 0$ such that $f(\gamma) = 0$.
\end{lem}
\begin{prop} \label{asymp expo}
Almost surely, it holds that
\begin{align*}
\lim_{k \rightarrow \infty} \frac{1}{k} \log \sum_{|\ii| = n(k)} \left( m_{\ii} r_{\ii} \right)^{x} = f(x), ~~~ x > 0.
\end{align*}
\end{prop}
\begin{proof}
This proposition follows from \eqref{sum log scale asymptotics}.
\end{proof}
\begin{thm}
The spectral exponent is given by the unique solution $\gamma > 0 $ of
\begin{align*}
f( \gamma) = 0,
\end{align*}
where $f$ is defined as in Lemma \ref{existence}.
\end{thm}
\begin{proof}
By Lemma \ref{existence}, the solution exists and is unique. Therefore, we have to show that
\begin{align*}
\lim_{t \rightarrow \infty} \frac{\log N_D^{(I)}(t)}{\log t} = \gamma ~~~ a.s.
\end{align*}
To this end, we define for $|\ii| = n(k)$ 
\begin{align*}
\tau_x(\ii) \coloneqq \frac{\left(r_{\ii}m_{\ii} \right)^x}{\sum_{|\jj| = n(k)} \left(r_{\jj}m_{\jj} \right)^x}.
\end{align*}
By Proposition \ref{asymp expo} we have for $x > \gamma$ (i.e. $f(x) < 0$) for $\epsilon > 0$ small enough that for all $c>0$ there exists $k_0 \in \mathbb{N}$ such that
\begin{align}
\tau_x(\ii) \ge \left(r_{\ii} m_{\ii} \right)^x e^{-k(f(x) + \epsilon)} \ge c \left(r_{\ii} m_{\ii} \right)^x, ~~~ \text{ for all } k \ge k_0. \label{estimate tau}
\end{align}
Since
\begin{align*}
\tau_x(\ii) &= \frac{\sum\limits_{|l|=n(k+1) - n(k), \atop l \in \theta_{\ii}I} \left(r_{\ii l}m_{\ii l} \right)^x }{\sum\limits_{|\jj| = n(k)} \left(r_{\jj}m_{\jj} \right)^x \sum\limits_{|l|=n(k+1) - n(k), \atop l \in \theta_{\ii}I} \left(r_{l}m_{ l} \right)^x} \\[8pt]
&= \frac{\sum\limits_{|l|=n(k+1) - n(k), \atop l \in \theta_{\ii}I} \left(r_{\ii l}m_{\ii l} \right)^x }{\sum_{|\jj| = n(k+1)} \left(r_{\jj}m_{\jj} \right)^x},
\end{align*}
where the second equality holds because $\theta_{\ii_1} I = \theta_{\ii_2} I$ for all $|\ii_1| = |\ii_2| = n(k)$, we have for every cut set $\Lambda$
\begin{align*}
\sum_{\ii \in \Lambda} \tau_x(\ii) = 1
\end{align*}
and thus, since $\Lambda_k$ is a cut set, we receive by Lemma \ref{estimate rinf minf}, for some $x' > 0$ and all $k \ge k_0$,
\begin{align*}
1 = \sum_{\ii \in \Lambda_k} \tau_x(\ii) \ge \sum_{\ii \in \Lambda_k} c \left(r_{\ii} m_{\ii} \right)^x \succeq c M_k k^{-x x'} e^{-k x}. 
\end{align*}
Therefore,
\begin{align}
M_k \preceq c k^{x x'} e^{k x} ~~~ a.s. \label{asympt M}
\end{align}
For $t > 1 $ large enought, let $k$ be such that  $t \in (e^{k-1},e^k]$. By Lemma \ref{estimate rinf minf} we then have $t \le T_k$. Together with \eqref{asympt M} and Lemma \ref{estimate DECF 3}, 
\begin{align*}
\frac{\log N_D^{(I)} (t)}{\log t} \le \frac{\log N(T_k)}{\log t} \le \frac{\log(cM_k)}{k-1} \preceq x ~~~ a.s.
\end{align*}
Since this holds for all $x > \gamma$, it follows
\begin{align*}
\frac{\log N_D^{(I)}(s)}{\log s} \preceq \gamma ~~~ a.s.
\end{align*}
Now, let $x < \gamma$ (i.e. $f(x) > 0$). For $\epsilon > 0$ small enough we have for some $k_0 \in \mathbb{N}$, analogously to the estimates in \eqref{estimate tau},
\begin{align*}
1 = \sum_{\ii \in \Lambda_k} \tau_x(\ii) \le \sum_{\ii \in \Lambda_k} c \left(r_{\ii} m_{\ii} \right)^x \le c M_k e^{-kx}, ~~~ \text{ for all } k \ge k_0
\end{align*}
and thus
\begin{align*}
M_k \ge c e^{kx}, ~~~ \text{ for all } k \ge k_0. 
\end{align*}
From Lemma \ref{estimate DECF 3}, we have
\begin{align}
\frac{\log N_D^{(I)}(k^\alpha T_k)}{k} \ge \frac{\log M_k}{k} \succeq x ~~~ a.s.  \label{asympt lower bound}
\end{align}
for some $\alpha > 0$. For $t > 1$ large enough and $k$ such that $t \in  (e^{k-1}, e^k]$ we have again from Lemma \ref{estimate rinf minf} for some $\alpha' > 0$
\begin{align*}
k^\alpha T_k \preceq k^{\alpha'} e^k \le e(1+\log t)^{\alpha'} t ~~~ a.s.
\end{align*}
and thus
\begin{align*}
\liminf_{k \rightarrow \infty} \frac{\log N_D^{(I)}(k^\alpha T_k)}{k} \le \liminf_{t \rightarrow \infty} \frac{\log N_D^{(I)}(e(1+\log t)^{\alpha'} t)}{\log t} ~~~ a.s.
\end{align*}
Since 
\begin{align*}
\lim_{t \rightarrow \infty} \frac{\log e(1+\log t)^{\alpha'} t}{\log t} = 1, ~~~ \lim_{t \rightarrow \infty } e(1+\log t)^{\alpha'} t = \infty,
\end{align*}
we have
\begin{align*}
\liminf_{k \rightarrow \infty} \frac{\log N_D^{(I)}(k^\alpha T_k)}{k} \le \liminf_{t \rightarrow \infty } \frac{\log N_D^{(I)}(t)}{ \log t} ~~~ a.s.
\end{align*}
Since \eqref{asympt lower bound} holds for all $x < \gamma$ we then receive
\begin{align*}
\frac{\log N_D^{(I)}(s)}{\log s} \succeq \gamma,  ~~~ a.s.
\end{align*}
\end{proof}

\begin{remark}
With the inequality
\begin{align*}
N_D^\mu(x) \le N_N^\mu(x) \le N_D^\mu(x) + 2, ~~~ x \ge 0,
\end{align*}
for arbitrary finite atomless Borel measure $\mu$ (see \cite[Proposition 5.5]{Fre04/05}), we also receive
\begin{align*}
\lim_{t \rightarrow \infty}  \frac{\log N_N^{(I)}(t)}{\log t} = \gamma, ~~~ a.s.
\end{align*}
\end{remark}

\end{document}